\documentclass[a4paper,12pt, reqno]{amsart}
\usepackage{amsmath, amsthm, amscd, amsfonts, amssymb, graphicx, color}
\usepackage[bookmarksnumbered, colorlinks, plainpages]{hyperref}
\hypersetup{colorlinks=true,linkcolor=red, anchorcolor=green, citecolor=cyan, urlcolor=red, filecolor=magenta, pdftoolbar=true}

\usepackage{tensor}
\usepackage{fullpage}

\usepackage[textwidth=20cm,top=2cm,bottom=1.5cm,left=2cm,right=2cm]{geometry}

\newtheorem{theorem}{Theorem}[section]
\newtheorem{lemma}[theorem]{Lemma}
\newtheorem{proposition}[theorem]{Proposition}

\theoremstyle{definition}
\newtheorem{definition}[theorem]{Definition}
\newtheorem{example}[theorem]{Example}

\newtheorem{conjecture}[theorem]{Conjecture}

\newtheorem{problem}[theorem]{Problem}
\newtheorem{remark}[theorem]{Remark}
\numberwithin{equation}{section}

\begin{document}
	
	\setcounter{page}{0}
	
	\title[Quasi-uniform type spaces]{Quasi-uniform type spaces}

	\author[Ya\'e Ulrich Gaba]{Ya\'e Ulrich Gaba$^{1,2,3,\dagger}$}

	\address{$^{1}$ Department of Mathematical Sciences, North West University, Private Bag
		X2046, Mmabatho 2735, South Africa.}

	\address{$^{2}$ Institut de Math\'ematiques et de Sciences Physiques (IMSP), 01 BP 613 Porto-Novo, B\'enin.}

	\address{$^{3}$ African Center for Advanced Studies (ACAS),
		P.O. Box 4477, Yaounde, Cameroon.}

	\email{\textcolor[rgb]{0.00,0.00,0.84}{yaeulrich.gaba@gmail.com
	}}


	
	
	\date{Received: xxxxxx; Accepted: zzzzzz.
		\newline \indent $^{\dagger}$Corresponding author}
	
	\begin{abstract}
		
		In this article we introduce and investigate the concept of partial quasi-metric type space as a generalization of both partial quasi-metric and quasi-metric type spaces. We show that many
		important constructions studied in K\"unzi's theory of partial quasi-metrics can be successfully extended to these spaces. In particular, we prove that the basic theories of topology
		and quasi-uniformity are essentially the same for quasi-metric type spaces as for quasi-metric spaces and by extensions, to partial quasi-metric type spaces. We also prove that the Banach, Kannan, Reich and Chatterjea fixed
		point theorems can be successfully extended to this more general setting.
	
	\end{abstract} 
	
	\maketitle
	
	\section{Introduction}
	
	Partial metric spaces (PMS) were introduced by Matthews\cite{mat} in 1992 where he explained that such functions can be used to
	study non-Hausdorff topologies. They generalize the concept of a metric space and are also useful in modelling partially defined
	information, which often appears in computer science. In fact, (complete) partial metric spaces constitute a suitable framework to model several examples
	of the theory of computation and also to model metric spaces via domain theory (see for instance \cite{heck,sal}). The particularity of these spaces is the property that the self-distance of any point of the space may not be zero. Recently, many authors have focused on the PMSs and their
	topological properties as well as
	fixed point results in these spaces (see e.g. \cite{altun,oltra}). 
	In \cite{kun}, by dropping the symmetry condition in the definition of a partial metric, K\"unzi et al. studied another variant of
	partial metrics, namely partial quasi-metrics and proved that they are equivalent to weighted quasi-metrics. On another hand, Gaba et al. \cite{gaba} introduced the so-called quasi-pseudometric type spaces as a generalization of the quasi-pseudometric spaces and majorly made use of the concept of quasi-cone metric space. Quasi-pseudometric type relax the triangle inequality. Indeed 
	establishing triangle inequality for quasi-pseudometric is often challenging but proving triangularity for quasi-pseudometric type can be much easier. In Section 3 of \cite{gaba2}, the author discussed some topological properties of quasi-pseudometric type spaces. For instance the concepts of left $K$-Cauchy sequence, right $K$-Cauchy sequence, $D^s$-Cauchy sequence, and convergence and completeness for a quasi-pseudometric type space are defined in a similar way as defined for a quasi-pseudometric space but naturally present a wider framework for topological studies. In particluar, we shall say that the $\alpha$-quasi-pseudometric $(X,D,\alpha)$ is bicomplete if the metric type space $(X,D^s,\alpha)$ is complete (see \cite[Definition 32]{gaba}).
	In this paper we aim at unifying both the concept of a quasi-metric type and that of a partial
	quasi-metric spaces by introducing the partial quasi-metric type space. In particular, we show that the basic theories of topology
	and quasi-uniformity are essentially the same for  quasi-metric type spaces as for quasi-metric spaces and we also find that Banach contractions, Kannan contractions, Reich contractions and Chatterjea contractions can easily be expressed in this new setting.
Also, in \cite{kun}, K\"unzi et al. described a bijection between quasi-metrics with weight and lopsided partial quasi-
metrics on $X$. In the present manuscript, we prove that a similar correspondence holds between partial quasi-metric type and weighted quasi-metric type.

\vspace*{0.5cm}

\section{Background definitions and first results}

First, we recall some definitions from the theory
of quasi-metric type spaces.

	\begin{definition}( Compare \cite[Definition 29]{gaba})
		Let $X$ be a nonempty set, and let the function $D:X\times X \to [0,\infty)$ satisfy the following properties:
		\begin{itemize}
			\item[(D1)] $D(x,x)=0$ for any $x \in X$;
			\item[(D2)] $D(x,y) \leq \alpha \big( D(x,z)+D(z,y) \big)$ for any points $x,y,z\in X$ and some constant $\alpha\geq 1$.
		\end{itemize}
		The triplet $(X,D,\alpha)$ is called a \textbf{quasi-pseudometric type space} or $\alpha$-quasi-pseudometric space.
Sometimes the constant $\alpha$ could be implied and then omitted and we could just write $(X,D)$ to refer to an $\alpha$-quasi-pseudometric space.  		
Moreover, if $D(x, y) = 0 = D(y, x) \Longrightarrow x = y$, then $D$ is said
to be a $T_0$-quasi-pseudometric type space or a quasi-metric type space  or an $\alpha$-quasi-metric. The latter condition is referred to as the $T_0$-condition.

We shall speak of {\bf extended} $\alpha$-quasi-pseudometric when the mapping $D$ can attain the value $\infty$.	
\end{definition}

	Let $(X,D,\alpha)$ be a quasi-pseudometric type space, the \textbf{\textit{conjugate}}(or \textbf{\textit{dual}}) of $D$ is the function denoted $D^{-1}$ and defined, whenever $x,y \in X$ 
	by
	$$D^{-1}(x,y)=D(y,x).$$ 
	
One can easily verify that, from a $T_0$-quasi-pseudometric type $D$ we obtain a metric type (in the sense of Khamsi \cite{khamsi} ) or $b$-metric $D^s$ by setting $D^s(x,y):= \max \{D(x,y),D(y,x)\}$ whenever $x,y \in (X,D,\alpha)$. More on topological properties of metric type spaces can be read in \cite{gab,khamsi}. 
 Also, if
$D$ is an $\alpha$-quasi-pseudometric then it is a $\beta$-quasi-pseudometric for each real $\beta \geq \alpha$.
Moreover, for $\alpha = 1$, we recover the
classical quasi-pseudometric; i.e. every quasi-pseudometric is a $1$-quasi-pseudometric, hence quasi-pseudometric type
generalizes quasi-pseudometric. Indeed there are $\alpha$-quasi-pseudometrics which are not quasi-pseudometrics.

\begin{example}
Consider $D:\mathbb{R} \times \mathbb{R} \to \mathbb{R}$ defined by $D(x, y) = \max\{0, (x-y)^2\}$ . Then $D$ is
a $2$-quasi-pseudometric on $\mathbb{R}$ because for every two real numbers $a$ and $b$, we have $$(a + b)^2 \leq 2 (a^2+b^2).$$ But $D$ is certainly not a quasi-pseudometric as $$D(1,-1) = 4 \nleq D(1, 0)+D(0, -1) = 1+1 = 2.$$ More generally, for every even integer $n=2k$, $D(x, y) = \max\{0, (x-y)^n\}$ is a $2^n$-quasi-pseudometric on $\mathbb{R}$. 

\vspace*{0.3cm}

A direct consequence of the above is that, for any function $f: X\to [0,+\infty)$, defined on a non-empty set $X$, the application $D_f$ defined on $X^2$ by $D_f(x,y)= \max\{0, (f(x)-f(y))^n\}$ is a $2^n$-quasi-pseudometric, and is a $T_0$-quasi-pseudometric type space if and only if $f$ is one-to-one.

\end{example}

\begin{example}
	A larger class of $\alpha$-quasi-pseudometrics can also be obtained by considering positive powers of quasi-pseudometrics. Indeed, if $(X,d)$ is be a quasi-pseudometric space, then the mapping $D$ defined by $D(x,y) = (d(x,y))^p $ whenever $x,y \in X$ and $p>1$ is a $2^{p-1}$-quasi-pseudometric.
\end{example}

\begin{example}
Let $(X,d)$ is be a quasi-pseudometric space. Let $\beta >1, \lambda\geq 0$, and $\mu > 0,$ and for $x, y \in X$ define $H (x, y) = \lambda d(x, y) + \mu d (x, y)^\beta$. In general $H$ is not a quasi-metric on $X$. However, for any $z\in X$, one has

\begin{align*}
	H(x,y)& = \lambda  d(x, y) + \mu d(x, y)^\beta \\
	    & \leq \lambda[d(x,z)+d(z,y)] + \mu[d(x,z)+d(z,y)]^\beta\\
	    & \leq \lambda[d(x,z)+d(z,y)] + 2^{\beta-1}\mu [d(x,z)^\beta+d(z,y)^\beta]\\
	    &\leq 2^{\beta-1} [H(x,z)+H(z,y)],
\end{align*}
i.e. $(X,H)$ is a $2^{\beta-1}$-quasi-metric space.
\end{example}	
	 
In order to make the reader more comfortable, we recall some definitions that we take from \cite{gaba}, as we shall need them later on.

\begin{definition}
	Let $(X,d,\alpha)$ be a quasi-pseudometric type space. For $x \in X$ and $\varepsilon > 0$, $$B_{d}(x,\varepsilon)=\lbrace y \in X: d(x,y) < \varepsilon \rbrace $$ denotes the open $\varepsilon$-ball at $x$. The collection of all such balls yields a base for a topology $\tau (d)$ induced by $d$ on $X$. Similarly, for $x \in X$ and $\varepsilon \geq 0$, $$C_{d}(x,\varepsilon)=\lbrace y \in X: d(x,y) \leq \varepsilon \rbrace $$ denotes the closed $\varepsilon$-ball at $x$.
	
\end{definition}

\begin{definition} 
	Let $(X,d,\alpha)$ be a quasi-pseudometric type space. The convergence of a sequence $(x_n)$ to $x$ with respect to $\tau(d)$, called \textbf{$d$-convergence} or \textbf{left-convergence} and denoted by $x_n \overset{d}{\longrightarrow} x$, is defined in the following way
	\begin{equation}
	x_n \overset{d}{\longrightarrow} x \Longleftrightarrow d(x_n,x) \longrightarrow 0 .
	\end{equation}
	
	Similarly, the convergence of a sequence $(x_n)$ to $x$ with respect to $\tau(d^{-1})$, called \textbf{$d^{-1}$-convergence} or \textbf{right-convergence} and denoted by $x_n \overset{d^{-1}}{\longrightarrow} x$,
	is defined in the following way
	\begin{equation}
	x_n \overset{d^{-1}}{\longrightarrow} x \Longleftrightarrow d(x,x_n) \longrightarrow 0 .
	\end{equation}
	
	Finally, in a quasi-pseudometric type space $(X,d,\alpha)$, we shall say that a sequence $(x_n)$ \textbf{$d^s$-converges} to $x$ if it is both left and right convergent to $x$, and we denote it as $x_n \overset{d^{s}}{\longrightarrow} x$ or $x_n \longrightarrow x$ when there is no confusion.
	Hence
	\[
	x_n \overset{d^{s}}{\longrightarrow} x \ \Longleftrightarrow \  x_n \overset{d}{\longrightarrow} x \ \text{ and }\ x_n \overset{d^{-1}}{\longrightarrow} x.
	\]
	
\end{definition}

\begin{definition}
	A sequence $(x_n)$ in a quasi-pseudometric type $(X,d,\alpha)$ is called
	\begin{itemize}

		\item[(a)] \textbf{left $K$-Cauchy} if for every $\epsilon >0$, there exists $n_0 \in \mathbb{N}$ such that 
		$$ \forall \  n,k: n_0\leq k \leq n \quad d(x_n,x_k )< \epsilon  ;$$
		
		\item[(b)] \textbf{ $d^s$-Cauchy} if for every $\epsilon >0$, there exists $n_0 \in \mathbb{N}$ such that 
		$$ \forall n, k \geq n_0 \quad d(x_n,x_k )< \epsilon .$$
	\end{itemize}

	Dually, we define in the same way, \textbf{right $K$-Cauchy} sequences.
\end{definition}

Observe that a sequence is $d^s$-Cauchy in the $T_0$-quasi-pseudometric type $(X,d,\alpha)$ if and only if it is Cauchy in the metric type $(X,d^s,\alpha)$.

\begin{definition}
	A quasi-pseudometric type space $(X,d,\alpha)$ is called
	
	\begin{itemize}
		\item \textbf{left $K$-complete} provided that any left $K$-Cauchy sequence is $d$-convergent,
		\item {\bf left Smyth sequentially complete} if any left $K$-Cauchy sequence is $d^s$-convergent.
	\end{itemize} 
\end{definition}
The dual notions of \textbf{right-completeness} are easily derived from the above. 

\begin{definition}
	A $T_0$-quasi-pseudometric type space $(X,d,\alpha)$ is called \textbf{bicomplete} provided that the metric type $d^s$ on $X$ is complete.
\end{definition}
	
The abbreviation $\alpha$-QPM refers to $\alpha$-quasi-pseudometric. In the coming section, we shall mention a few properties on the topology of $\alpha$-QPM spaces.

	\section{Topology of $\alpha$-QPM}	
	The investigations on topological properties of $\alpha$-QPM have already began in \cite{gaba} but was just limited to the ideas of $K$-Cauchy sequences, left(right) convergence and that of bicompleteness. Here we shall focus more on (quasi)metrizability of these spaces.

	 In the following lines, our purpose is to show that basic theories of topology
	and quasi-uniformity are essentially the same for quasi-metric type spaces as for quasi-metric spaces. We recall that given a set $X \neq \emptyset$, a \textbf{quasi-uniformity} $\mathcal{U}$ on $X$ is a filter on $X\times X$ such that
	\begin{itemize}
		\item[(1)] Each member $U$ of $\mathcal{U}$ contains the diagonal $\Delta_X=\{ (x,x):x\in X\}$ of $X$;
		\item[(2)] For each member $U$ of $\mathcal{U}$ there exists a $V \in \mathcal{U}$ such that $V^2  \subseteq U$.
	\end{itemize}
		Here $$V^2:=V\circ V=\{(x,z) \in X\times X: \text{ there exixts } y\in X \text{ such that } (x,y)\in V \text{ and }(y,z)\in V\}.$$
	
	The members $U$ of $\mathcal{U}$ are called \textbf{entourages} of $\mathcal{U}$ and the pair $(X,\mathcal{U})$ a \textbf{quasi-uniform space}.
	
Each quasi-uniformity $\mathcal{U}$ on a set $X$ induces a topology $\tau(\mathcal{U})$ as follows: For each $x\in X$ and $U\in \mathcal{U}$, set $$U(x)=\{ y\in X: (x,y)\in U\}.$$ A subset $G\subseteq X$ belongs to $\tau(\mathcal{U})$ if and only if for each $x\in G$, there exists $U\in \mathcal{U}$ such that $U(x)\subseteq G$. In particular, we know from \cite{kun2} that given a quasi-pseudometric $d$ on a set $X$, the filter on $X\times X$ generated by the base $\{U_\epsilon:\epsilon>0\}$ where $U_\epsilon= \{(x,y)\in X\times X:d(x,y)<\epsilon \}$, is a quasi-uniformity called \textbf{quasi-pseudometric quasi-uniformity} and denoted $\mathcal{U}_d$. It is the quasi-uniformity induced by $d$ on $X$. Indeed, just observe that for each $\epsilon>0$, $U^2_{\epsilon/2} \subseteq U_\epsilon$. 

\vspace*{0.3cm}

In the sequel, we prove that given an $\alpha$-quasi-pseudometric $D$ on a set $X$, the filter on $X\times X$ generated by the base $\{N_\epsilon:\epsilon>0\}$ where $N_\epsilon= \{(x,y)\in X\times X:D(x,y)<\epsilon \}$, is also a quasi-uniformity that we shall call \textbf{$\alpha$-quasi-uniformity}.   

\begin{lemma}\label{lem3.1}
For any $\alpha$-quasi-pseudometric space $(X, D,\alpha)$, the filter $\mathcal{U}_D$ on $X\times X$ generated by the base $\{N_\epsilon:\epsilon>0\}$ where $N_\epsilon= \{(x,y)\in X\times X:D(x,y)<\epsilon \}$, is a quasi-uniformity.   
\end{lemma}

	\begin{proof}
		If $U \in \mathcal{U}_D$, then, there exists $\epsilon>0$ such that $N_\epsilon \subseteq U.$ Since for each $x\in X, D(x,x)=0$, then each $(x,x)\in N_\epsilon,$ so $\Delta_X \subseteq U.$
		Moreover, if $U \in \mathcal{U}_D$, then find $\epsilon>0$ such that $N_\epsilon \subseteq U.$ Hence, observe that there is an $r>0$ such that $2\alpha r \leq \epsilon.$ Set $V = N_r$ and note that for $(x,y),(y,z) \in N_r$ then $D(x,z)\leq \alpha(D(x,y)+D(y,z)\leq 2 \alpha r \leq \epsilon$. So $N_s \circ N_s \subseteq N_\epsilon \subseteq U.$

	\end{proof}

The next lemma establishes that any $\alpha$-quasi-pseudometric is topologically equivalent to a bounded $\alpha$-quasi-pseudometric.

\begin{lemma}\label{lemma2.1}
	Let $(X,D,\alpha)$ be an $\alpha$-quasi-pseudometric space and let $\lambda \in (0,1)$. Then there exists an $\alpha$-quasi-pseudometric $E$ on $X$ such that $E(x,y)\leq \lambda$ for all $x,y \in X$ and $E$ and $D$ induce the same topology on $X$.
\end{lemma}

\begin{proof}
	We define $E(x, y) = \min\{\lambda, D(x, y)\}$ and claim that $E$ is an $\alpha$-QPM on $X$. The properties (D1) is immediate from the definition. For the property (D2) (relaxed triangle inequality ), consider the points $x,y,z\in X$. Then $E(x, y) \leq \lambda$ and so $$E(x, y) \leq\big( E(x,z)+E(z,y) \big)$$
	when either $E(x,z)=\lambda$ or $E(z,y)=\lambda$.
	
The only remaining case is when $E(x, z) = D(x, z) < \lambda$ and  $E(z,y)=D(z,y)<\lambda$. But 
$$D(x,y) \leq \alpha \big( D(x,z)+D(z,y) \big)$$
and $ E(x, y) \leq D(x, y)$, and so

$$E(x,y) \leq \alpha \big( E(x,z)+E(z,y) \big).$$

Thus $E$ is an $\alpha$-QPM on $X$. It only remains to show that the topology induced by $E$ is the
same as the one induced by $D$. We have, as $n\to \infty$, that

$$ E(x_n , x) \to 0 \quad \Leftrightarrow \quad \min\{\lambda, D(x_n , x)\} \to 0 \quad \Leftrightarrow \quad D(x_n , x) \to 0$$
and we are done.
\end{proof}

The $\alpha$-QPM $E$ in the above lemma is said to be bounded by $\lambda$.

Next, we establish a connexion between left(right)-$K$-completeness and the bicompleteness.

\begin{proposition}
		Let $(X,D,\alpha)$ be an $\alpha$-quasi-pseudometric space. Then $(X,D,\alpha)$ is bicomplete if and only if it is both left $K$-complete and right $K$-complete and for any sequence $(x_n)_{n\geq 1}\subseteq X$ which both left-convergent and right-convergent, the limits coincide.
\end{proposition}

	\begin{proof}
	For the sufficient condition, let  $(x_n)_{n\geq 1}$ be a sequence of elements of the bicomplete $\alpha$-quasi-pseudometric space $(X,D,\alpha)$ and assume that $(x_n)_{n\geq 1}$ is both a left $K$-Cauchy sequence and a right $K$-Cauchy sequence. That is, 
	
	$\epsilon >0$, there exists $n_0,n_1 \in \mathbb{N}$ such that 
	$$ \forall \  n,k: n_0\leq k \leq n \quad D(x_n,x_k )< \epsilon  ;$$
	
	and 
	$$ \forall \  n,k: n_1\leq k \leq n \quad D(x_k,x_n )< \epsilon  .$$
Hence there exists $n_{\max}=\max\{n_0,n_1\}$ such that 
$$ \forall \  n,k\geq n_{\max} \quad D^s(x_n,x_k )< \epsilon  ,$$	
i.e. $(x_n)_{n\geq 1}$ is $D^s$-Cauchy sequence and
since $(X,D,\alpha)$ is bicomplete, there exists $x^*$ such that $x_n \overset{D^{s}}{\longrightarrow} x^*.$
This implies, by definition that
$x_n \overset{D}{\longrightarrow} x^* \ \text{ and }\ x_n \overset{D^{-1}}{\longrightarrow} x^*,$ i.e. $(X,D,\alpha)$ is both left-$K$-complete and right-$K$-complete. Moreover, the sequence $(x_n)_{n\geq 1}$ is left-convergent and right-convergent to the limit $x^*.$

Conversely, assume that $(X,D,\alpha)$ is both left-$K$-complete and right-$K$-complete and let $(x_n)_{n\geq 1}$ be a $D^s$-Cauchy sequence in $(X,D,\alpha)$.
 Hence, 

$\epsilon >0$, there exists $n^*\geq 1$ such that 
$$ \forall \  n,k\geq n^* \quad D^s(x_n,x_k )< \epsilon  .$$	

This means that $\epsilon >0$, there exists $n^* \in \mathbb{N}$ such that 
$$ \forall \  n,k: n^*\leq k \leq n \quad D(x_n,x_k )< \epsilon  ;$$

and 
$$ \forall \  n,k: n^*\leq k \leq n \quad D(x_k,x_n )< \epsilon  ,$$
i.e. the sequence  $(x_n)_{n\geq 1}$ is both is both a left $K$-Cauchy sequence and a right $K$-Cauchy sequence. Since $(X,D,\alpha)$ is both left $K$-complete and right $K$-complete, there exist $a^*,b^*$ such that $x_n \overset{D}{\longrightarrow} a^* \ \text{ and }\ x_n \overset{D^{-1}}{\longrightarrow} b^*$ and by assumption $a^*= b^*.$ Therefore
 $x_n \overset{D}{\longrightarrow} a^* \ \text{ and }\ x_n \overset{D^{-1}}{\longrightarrow} a^*$ which is equivalent to $x_n \overset{D^s}{\longrightarrow} a^* ,$ i.e. $(X,D,\alpha)$ is bicomplete.
 
This completes the proof.

	\end{proof}

\begin{definition}
	A topological space is called a  quasi-metrizable type
	space if there exists a quasi-pseudometric type $D$ inducing the given topology
	on it.
\end{definition}


We have the following interesting lemma 

\begin{proposition}\label{lemma2.5}
	Quasi-metrizability type is preserved under countable Cartesian product provided the series of coefficients is convergent\footnote{Actually, it is enough to have a uniform bound.}.
\end{proposition}

\begin{proof}

Let $\{ (X_n,D_n,K_n), n\in \mathbb{N}\}$ be a collection of quasi-metrizable type spaces. Let $\tau_n = \tau(D_n)$ the topology induced by $D_n$ on $X_n$. Set $$(X,\tau) = \prod_{n}(X_n,\tau_n),$$

where $X$ is the Cartesian product of $X_n, n\geq 1$ and $\tau$ the product topology. We have to prove that there is a quasi-metric type $D$ on $X$ which induces the topology $\tau$.

By Lemma \ref{lemma2.1}, we may suppose that $D_n$ is bounded by $2^{-n}$, else we
replace $D_n$ by another quasi-metric type which induces the same topology and which is bounded by $2^{-n}$. For $x,y\in X$, recall that $x=(x_n)_{n\geq 1}$ and $y=(y_n)_{n\geq 1}$. Now define, on $X^2$, the function $D$, as 

$$D(x,y) = \sum_{n\geq 1}D_n(x_n,y_n).$$

Observe that $D$ is well defined since and not extended, since 

$$D(x,y) = \sum_{n\geq 1}D_n(x_n,y_n)\leq \sum_{n\geq 1} 2^{-n} <\infty.$$
Also $D$ is a quasi-metric type on $X$.
Indeed, we clearly have 

$$x = y \iff x_n= y_n \quad \forall n \iff D_n(x_n,y_n) =0=D_n(y_n,x_n) \quad \forall n \iff D(x,y)=0=D(y,x).$$ 

Moreover, for $x=(x_n)_{n\geq 1}, y=(y_n)_{n\geq 1}, z=(z_n)_{n\geq 1} \in X$, we have 

\begin{align*}
 D(x,y) &= \sum_{n\geq 1}D_n(x_n,y_n) \\
        & \leq \sum_{n\geq 1}K_n[D_n(x_n,z_n) + D_n(z_n,y_n)]\\
        &\leq \left[\sum_{n\geq 1}K_n\right] \left(\sum_{n\geq 1}[D_n(x_n,z_n) + D_n(z_n,y_n)]\right)\\
        & = \left( \left[\sum_{n\geq 1}K_n\right]\sum_{n\geq 1}D_n(x_n,z_n) \right)+ \left(\left[\sum_{n\geq 1}K_n\right] \sum_{n\geq 1}D_n(z_n,y_n)\right)\\
        &= K (D(x,z)+D(z,y)
\end{align*}

 because each $D_n$ is of a quasi-metric type and $\sum_{n\geq 1}K_n =K<\infty.$
 
 \vspace*{0.2cm}
 
Let $\mathcal{U}$ be the topology induced by the quasi-metric type $D$.

\vspace*{0.2cm}

 \underline{Claim:} that $\mathcal{U}$ coincides with $\tau.$

\vspace*{0.2cm}

For $x\in X$ and $x\in G \in \mathcal{U}$, there exists $r>0$ such that $$B_D(x,r):=\{ y\in X, D(x,y)<r\} \subset G.$$

Choose $N_0\in \mathbb{N}$ such that $\sum_{n=1}^{N_0} 2^{-n}<\frac{r}{2}$. For each $n=1,2,\cdots N_0$, set $$V_n = B_{D_n}\left(x_,\frac{r}{2N_0}\right):= \left\lbrace y_n\in X_n, D_n(x_n,y_n)<\frac{r}{2N_0}\right\rbrace.$$ Now for $n>N_0$, set $V_n= X_n$ and put $V = \prod_{n\geq 1}V_n.$

Obviously, $x\in V$ and $V$ is an open set in the product topology $\tau$ on $X$. Moreover, for each $y\in V$

\begin{align*}
D(x,y) &= \sum_{n\geq 1}D_n(x_n,y_n) \\
       & = \sum_{n\geq 1}^{N_0}D_n(x_n,y_n) + \sum_{n\geq N_0+1}D_n(x_n,y_n)\\
       &\leq N_0 \frac{r}{2N_0} + \sum_{n\geq N_0+1} 2^{-n}\\
       & < \frac{r}{2} + \frac{r}{2} = r,
       \end{align*}
i.e. $V\subset B_D(x,r)\subset G$. Therefore $G$ is open in the product topology. 

\vspace*{0.2cm}

Conversely suppose $G$ is open in the product topology and let  $x=(x_n)_{n\geq 1}\in G$. Choose a standard basic open set $V$ such that $x\in V \subset G.$ Let $V = \prod_{n\geq 1}V_n$ where each $V_n$ is open in $X_n$ for $n=1,2,\cdots N_0$ and $V_n=X_n$ for all $n> N_0$. For $n=1,2,\cdots N_0$, let $$r_n = \inf \{ D_n(x_n,a_n): a_n \in X_n \setminus V_n \} \quad \text{ if } X_n\neq V_n$$

and $r_n = 2^{-n}$ otherwise. Now let $r = \min \{ r_1,r_2,\cdots, r_{N_0}  \}$. We claim that $B_D(x,r) \subset V$. Indeed for $y=(y_n)_{n\geq 1}\in B_D(x,r),$ we have $D(x,y)= \sum_{n\geq 1}D_n(x_n,y_n) <r$ and so $D_n(x_n,y_n)< r\leq r_n $ for $n=1,2,\cdots N_0$, which means $y_n V_n$ for $n=1,2,\cdots N_0$. Also for $n>N_0, y_n \in V_n=X_n$. Hence $y\in V$ and $B_D(x,r)\subset V\subset G.$ Therefore $G$ is open with respect to the quasi-metric type topology and $\tau\subset \mathcal{U}$. We have proved that the topologies $\tau$ and $\mathcal{U}$ coincide.

\end{proof}

\begin{remark}
One could define a topologically left(right) $K$-complete quasi-metrizable type
space if there exists a topologically left(right) $K$-complete quasi-pseudometric type $D$ inducing the given topology
on it. The author intends, in \cite{gabx} to investigate the features of  topologically left(right) $K$-complete quasi-metrizable type
space. 
\end{remark}

An interesting conjecture, that we intend to establish in \cite{gabx} is that for an $\alpha$-quasi-metric, the induced topology is quasi-metrizable. Indeed:

\begin{conjecture}(Compare \cite[Theorem 1]{aim})
	Let $(X,D,\alpha)$ be an $\alpha$-quasi-pseudometric space. Then there exists $0<\beta=\beta(\alpha)\leq 1$ such that 
	
	\[   \rho(x,y) = \inf \left\lbrace  \sum_{i=1}^{n}D^\beta(x_i,x_{i+1}):x=x_1,x_2,\cdots,x_{n+1}=y \in X, n\in \mathbb{N}\right\rbrace  \]
	is a quasi-pseudometric on $X$ with $\rho^{1/\beta}$ equivalent to $D$.
\end{conjecture}

\vspace*{0.2cm}

Rectangular ($b$)-metric spaces were presented by Branciari \cite{bran} and Shukla\cite{geo} and this suggest that one replaces the relaxed triangle
inequality in the definition of an  $\alpha$-QPM with the relaxed polygonal inequality, which asserts that there is a constant $K\geq 1$ such that for some $n\geq 2$, for all $x,y,$ and all distinct points $z_1,\cdots,z_{n-1} \in X\setminus\{x,y\},$

\[D(x,y) \leq K[d(x,z_1)+D(z_1,z_2)+\cdots D(z_{n-1},y)] .\]

This leads to the following definition.

\begin{definition}\label{relaxed}
	A relaxed $\alpha$-quasi-pseudometric $D$ (of order $n$) is an $\alpha$-quasi-pseudometric $D$ for which $D$ satisfies the relaxed polygonal inequality, that is 
	\[D(x,y) \leq K[d(x,z_1)+D(z_1,z_2)+\cdots D(z_{n-1},y_)] \]
	for all $x,y,$ and all distinct points $z_1,\cdots,z_{n-1} \in X\setminus\{x,y\}.$
\end{definition}

\begin{remark}
	In view of the above definition, one sees that if $(X,D)$ is a $T_0$ relaxed $\alpha$-quasi-pseudometric; for $K=1, n=3$, we obtain that $D^s(x,y)= \max\{D(x,y),D(y,x)\}$ is a rectangular metric in the sense of Branciari \cite{bran} and for $K\geq 1,n=3$, we obtain a rectangular $b$-metric in the sense of Shukla\cite{geo}. Moreover, for $K=1(\geq 1), n=2 $, the classical $D^s(x,y)= \max\{D(x,y),D(y,x)\}$ gives back the metric (metric type in the sense of Khamsi \cite{khamsi}). Furthermore,  if $(X,D)$ is a $T_0$ relaxed $\alpha$-quasi-pseudometric, then $D^s(x,y)= \max\{D(x,y),D(y,x)\}$ is a metric type in the sense of Gaba\cite{gab}. Note that if $(X,D)$ is a $1$-quasi-pseudometric, then it is a relaxed $1$-quasi-pseudometric with $n=2$ and also a relaxed $1$-quasi-pseudometric with $n\geq 3$ but the reverse is not true, since there are rectangular metric spaces (i.e. relaxed $1$-quasi-pseudometric with $n= 3$) which are not $1$-quasi-pseudometric.
\end{remark}

We have the following diagram where arrows stand for inclusions. The reverse
inclusions do not hold.

\begin{align*} 
\text{quasi-metric space} \quad  & \longrightarrow  \qquad   \quad   \alpha\text{-quasi-metric space} \\
  \downarrow   \qquad  \qquad  \qquad              &    \qquad  \qquad  \qquad \qquad    \downarrow   \\
        \text{relaxed quasi-metric space}  \quad &    \longrightarrow    \quad      \text{relaxed } \alpha\text{-quasi-metric space space}       
\end{align*}

\vspace*{0.2cm}

\begin{lemma}
	
	An $\alpha$-quasi-metric $D$ is a relaxed $\alpha^{n-1}$-quasi-metric (of order $n$).
\end{lemma}

\begin{proof}
	Let $(X,D)$ be an $\alpha$-quasi-metric space. Let $z_1,\cdots,z_{n-1} \in X\setminus\{x,y\}$ be distinct points. Then we have,
	
	\begin{align*}
		D(x,y) & \leq \alpha [D(x,z_1) + D(z_1,y)]\\
		       & \leq \alpha [D(x,z_1) + \alpha(D(z_1,z_2)+D(z_2,y)  ] \\
		       & = \alpha^2 D(x,z_1) + \alpha^2 D(z_1,z_2)+\alpha^2 D(z_2,y)  \\
		       & \leq \alpha^2 D(x,z_1)+ \alpha^2 D(z_1,z_2)+\alpha^2 [\alpha D(z_2,z_3)+\alpha D(z_3,y)  ]\\
		       & \leq \alpha^3 D(x,z_1)+ \alpha^3 D(z_1,z_2)+\alpha^3  D(z_2,z_3)+\alpha^3D(z_3,y)\\
		       & \vdots\\
		       & \leq \alpha^{n-1} D(x,z_1)+ \alpha^{n-1}  D(z_1,z_2)+\alpha^{n-1}   D(z_2,z_3)+\cdots +\alpha^{n-1} D(z_{n-1},y).
	\end{align*} 
	The proof is complete.
\end{proof}

We conclude this section by a few fixed point related results in $\alpha$-quasi-pseudometric spaces.

\begin{definition} (Compare \cite[Definition 33]{gaba})
	For a $K$-quasi-pseudometric space $(X, D,K)$, a function $f : X \to X$ is a Lipschitz
	map with bound $q$ if $q$ is such that for each $x, y \in X$, $D(f (x), f (y)) \leq q D(x, y)$. The smallest constant $q$ verifying the previous inequality will be denoted $Lip(f)$.

	Moreover, we shall say that $f$ is a
	contraction
	if it is a Lipschitz map with bound
	$q< \frac{1}{K^2}$.
\end{definition}

\begin{lemma}\label{lem2.2}
	Let $(X, D,K)$ be a $K$-quasi-pseudometric space and $f : X \to X$ a Lipschitz map with bound $q$ then $f$ is uniformly continuous.
\end{lemma}

\begin{proof}
	For each $r > 0$ and $x,\in X$, let $s =\frac{r}{q};$ then if $y\in X$ is such that $d(x, y) < s$ then we have $$d(f (x), f (y)) < qd(x, y) < qs = q\frac{r}{q} = r. $$
\end{proof}

\begin{lemma}\label{lem2.3}
	
	Let $(X, D,K)$ a $K$-quasi-pseudometric space and $f : X \to X$ a Lipschitz map with bound $q$. Then for each $x \in X$ and any $n \geq 0$, we have
	$$D(x,f^n(x)) \leq \left[\sum_{i=1}^{n} K^i q^{i-1}\right] D(x,f(x)).$$
	
\end{lemma}

\begin{proof}
	We shall use an inductive argument. Indeed for $n=0$ and $x \in X$, we trivially have
	\[ D(x,f^0(x)) = D(x,x)=0 \leq \left[\sum (\emptyset) \right] D(x,f(x)) .\]
	Assume the inequality for $n$ and all $x \in X$. Using the relaxed triangle inequality of the $K$-quasi-metric space and the inductive hypothesis, we write:
	
	\begin{align*}
	D(x,f^{n+1}(x)) & \leq K [D(x,f(x)) + D(f(x),f^{n+1}(x))] \\
	& \leq K D(x,f(x))  + K D(f(x),f^{n}f(x))\\
	&\leq K D(x,f(x)) + K \left[\sum_{i=1}^{n} K^i q^{i-1}\right]  D(f(x),f^2(x)) \\
	&\leq K D(x,f(x)) + K \left[\sum_{i=1}^{n} K^i q^{i-1}\right]q  D(x,f(x)) \\
	&\leq K D(x,f(x)) + K \left[\sum_{i=1}^{n} K^{i+1} q^{i}\right] D(x,f(x)) \\
	& =  K D(x,f(x)) + K \left[\sum_{i=2}^{n+1} K^{i} q^{i-1}\right] D(x,f(x)) \\
	& =\left[\sum_{i=1}^{n+1} K^i q^{i-1}\right] D(x,f(x)).
	\end{align*}
	
	So the inequality holds for $n + 1$ and arbitrary $x \in X$, completing our inductive argument.
\end{proof}

\begin{theorem}
Let $(X, d,\alpha)$ be a Hausdorff left $K$-complete $\alpha$-quasi-pseudometric such that $\alpha>1$. If $f:X\to X$ is a contraction, then $f$ has a fixed point.
\end{theorem}

\begin{proof}

	Let $x_0 \in X$ and inductively define the sequence $x_n= f(x_{n-1})$, hence for each $n\in \mathbb{N}, x_n = f^n(x_0).$ Now we show that $(x_n)$ is left $K$-Cauchy. Let $\varepsilon>0$; there exists $N_1$ such that $$\frac{1}{\alpha^{2N_1-2}}\frac{1}{\alpha-1}d(x_0,x_1)<\varepsilon.$$
	
	Now let $m\geq n\geq N_1$ and since $f$ is Lipschitz with bound $q< \frac{1}{\alpha^2}$ by Lemma \ref{lem2.3}, we have 
	
	 \begin{align*}
	 	d(x_n,x_m) = d(f^n(x_0),f^m(x_0) &\leq q^n d(x_0,f^{m-n}(x_0)) \leq q^n \left( \sum_{i=1}^{m-n}\alpha^i q^{i-1}\right)d(x_0,fx_0)\\
	 	&\leq \frac{1}{\alpha^{2n}}\left(\sum_{i=1}^{m-n}\frac{1}{\alpha^{i-2}}\right)d(x_0,x_1)\leq \frac{1}{\alpha^{2n}}\left(\sum_{i=1}^{\infty}\frac{1}{\alpha^{i-2}}\right)d(x_0,x_1)\\
	 	&=\frac{1}{\alpha^{2n-2}}\left(\sum_{i=1}^{\infty}\frac{1}{\alpha^{i}}\right)d(x_0,x_1)= \frac{1}{\alpha^{2n-2}}\frac{1}{\alpha-1}d(x_0,x_1)\\
	 	& \leq \frac{1}{\alpha^{2N_1-2}}\frac{1}{\alpha-1}d(x_0,x_1)< \varepsilon.
	 \end{align*}
This proves that $(x_n)$ is a left $K$-Cauchy sequence and since $(X,d,\alpha)$ is left $K$-complete, there exists some $x^*\in X$ such that $x_n \overset{d}{\longrightarrow} x^*$ and $x_{n+1} \overset{d}{\longrightarrow} x^*$. Thus by continuity of $f$ shown in Lemma \ref{lem2.2}, we have:

$$f(x_n) \overset{d}{\longrightarrow} f(x^*) \Longleftrightarrow x_{n+1} \overset{d}{\longrightarrow} f(x^*), $$
i.e. $f(x^*)=x^*$ since $(X,d,\alpha)$ is Hausdorff.
\end{proof}

For more results regarding fixed point theory in $K$-quasi-metric spaces, the interested reader could read the article by Gaba\cite{gaba2}. In that paper, the author mostly considered Hausdorff left-complete $T_0$-quasi-pseudometric type spaces.

\section{Main results}

We can now introduce our main new concept.

\begin{definition}\label{maindef} (Compare \cite[Definition 1.]{kun})
A partial quasi-metric type (or $K$-partial quasi-metric ) on a set $X$ is a function $p: X \times X \to [0, \infty)$ such that:
\begin{enumerate}
\item[(1a)] $p(x, x)\leq p(x, y)$ whenever $x, y \in X$,
\item[(1b)] $p(x, x)\leq p(y, x)$ whenever $x, y \in X$,
\item[(2)] $p(x, z) + p(y, y) \leq K( p(x, y) + p(y, z))$ whenever $x, y,z \in X$, for some $K\geq 1,$
\item[(3)] $x = y$ iff $(p(x, x) = p(x, y)$ and $p(y, y) = p(y, x))$ whenever $x, y \in X$.
\end{enumerate}

The triplet $(X, p,K)$ will be called $K$-partial quasi-metric space.

If $p$ satisfies all these conditions except possibly (1b), we shall speak of a \textit{lopsided partial quasi-metric type} or a \textit{lopsided $K$-partial quasi-metric}.
\end{definition}

\begin{remark}
If $p$ is a $K$-partial quasi-metric on $X$ satisfying
$(4)$ $p(x, y) = p(y, x)$ whenever $x, y \in X$, then $p$ is called a partial $K$-metric on $X$ in the sense of Shukla \cite{shu}. Moreover, similarly to Definition \ref{relaxed}, one could easily define a relaxed $\alpha$-partial quasi-metric $D$. 
\end{remark}

\begin{lemma}Compare \cite[Lemma 1.]{kun}
 Any lopsided $K$-partial quasi-metric $p$ on a set $X$ also satisfies the condition:
 
$(3')$ $x = y$ iff $(p(x, x) = p(y, x)$ and $p(y, y) = p(x, y))$ whenever $x, y \in X$.
\end{lemma}

\begin{lemma}Compare \cite[Lemma 2.]{kun}

\begin{enumerate}
	\item[(a)] Each $K$-quasi-metric $p$ on X is a $K$-partial quasi-metric on $X$ with $p(x, x) = 0$ whenever $x \in X$.
	\item[(b)] If $p$ is a $K$-partial quasi-metric (resp. a $K$-quasi-metric) on $X$, then so is its conjugate $p^{-1} (x, y) = p(y, x)$ whenever $ x, y \in X$.
	\item[(c)] If $p$ is a $K$-partial quasi-metric (resp. a $K$-quasi-metric) on $X$, then $p^+$ defined by $p^+(x, y) = p(x, y) + p^{-1}(x, y)$ is a partial $K$-metric (resp. a $K$-metric) on $X$.

\end{enumerate}

\end{lemma}

\begin{remark}
It is clear that every partial quasi-metric space is a $K$-partial quasi-metric
space with coefficient $K = 1$ and every $K$-quasi-metric space is a $K$-partial quasi-metric
space with the same coefficient and zero self-distance.
\end{remark}

\begin{example}
	Let $X =[0,\infty)$ , $p > 1$ a constant and $b : X \times X \to [0,\infty)$ be defined:
	$$ b(x,y) = [\max\{x,y\}]^p + [\max\{0,x-y\}]^p  \quad \text{for all } x,y \in X.$$
	Then (X, b) is a $K$-partial quasi-metric space with coefficient $K = 2^p$ but
	it is neither a $K$-quasi-metric nor a partial quasi-metric space. Indeed, for $x>0$, $b(x,x) = x^p \neq 0$; therefore, $b$ is not a $K$-quasi-metric on $X$. Also, for $x = 5, y = 2, z = 4$, we have $b(x, y)=5^p + 3^p$ and $b(x,z) + b(z,y)-b(z,z) = 5^p+1+4^p+2^p-4^p$, so $b(x,y) >b(x,z) + b(z,y)-b(z,z)$ for all $p > 1$; therefore, $b$ is not a partial quasi-metric on $X$.
	
\end{example}

The following proposition allows us to construct $K$-partial quasi-metrics from existing ones.

\begin{proposition} Compare \cite[Proposition 1.]{shu}
	Let $X$ be a nonempty set such that $p$ is a partial quasi-metric and $d$ an
	$\alpha$-quasi-metric with coefficient $\alpha > 1$ on $X$. Then the function $b : X \times X \to [0,\infty)$
	defined by $b(x, y) = p(x, y) + d(x, y)$ whenever $x, y \in X$ is an $\alpha$-partial quasi-metric on
	$X$.
\end{proposition}

\begin{proof}
	The axioms (1a), (1b) and (3) are easily verified for the function $b$. For the axiom (2), we have, for $x, y, z \in X$, the following:
	\begin{align*}
		b(x, y) & = p(x, y) + d(x, y)\\
		        & \leq p(x, z) + p(z, y) - p(z, z) + \alpha[d(x, z) + d(z, y)]\\
		        & \leq \alpha[p(x, z) + p(z, y) - p(z, z) + d(x, z) + d(z, y)] \\
		        & = \alpha[b(x, z) + b(z, y) - b(z, z)]\\
		        & \leq \alpha[b(x, z) + b(z, y)] - b(z, z).
	\end{align*}
	Therefore, (2) is also satisfied and so $b$ is an
	$\alpha$-partial quasi-metric on $X$.
\end{proof}

\begin{remark}
Note that, for any partial quasi-metric $p$ and any $q\geq 1$, the application $p^q$ is a $K$-partial quasi-metric ith coefficient $K = 2^{q-1}$. Indeed, observe that for any nonnegative rela numbers $a,b,c$ with $b\geq c$ and any $q\geq 1$ 
$$ \left(\frac{a+b}{2}\right)^p \leq \frac{a^p+b^p}{2} \quad \text{   and   
} \quad (b-c)^q\leq b^q - \frac{c^q}{2^{q-1}}.$$
\end{remark}

\section{Fixed point theory on $K$-PQM}

The notions such as convergence, completeness, Cauchy sequence in the setting of partial
metric spaces, can be found in \cite{alg,mat} and references therein.

For every $K$-partial quasi-metric space $(X,p,K)$, the collection of balls $$p(x,\epsilon) = \{y \in X : (x, y) < \epsilon + p(x, x)\}$$ yields a base for a $T_0$ topology $\tau(p)$ on $X$.

Also, it is easy to see that given a $K$-partial quasi-metric $p$ on a set $X$, the filter on $X\times X$ generated by the collection $\{P_\epsilon:\epsilon>0\}$ where $P_\epsilon= \{(x,y)\in X\times X:p(x,y)<\epsilon+p(x,x) \}$, is also a quasi-uniformity that we shall call \textbf{$\alpha$-partial quasi-uniformity}. Indeed for a $K$-partial quasi-metric space $(X,p,K)$ with $K=1$, for each $\varepsilon>0$, $P^2_{\varepsilon/2} \subseteq P_\varepsilon.$ 

Since for $(x,z) \in P^2_{\varepsilon/2}$, there exists $y\in X$ such that $p(x,y)\leq \varepsilon/2 + p(x,x) $ and $p(y,z)\leq \varepsilon/2 + p(y,y),$ i.e. $p(x,z) -p(x,x)\leq p(x,y) + p(y,z) -p(y,y)-p(x,x) <\varepsilon.$

\vspace*{0.3cm}

Now, we define Cauchy sequence and convergent sequence in $K$-partial quasi-pseudometric spaces.

\begin{definition}\label{deff1}
	Let $(X, p,K)$ be a $K$-partial quasi-metric space with coefficient $K$. Let $(x_n)_{n\geq 1}$ be any sequence in $X$ and $x \in X$. Then:
	
	\begin{enumerate}
		\item The sequence $(x_n)_{n\geq 1}$ is said to be convergent with respect to $\tau(p)$ (or $\tau(p)$-convergent) and converges to $x$, if $\lim\limits_{n\to \infty} p(x_n , x) = p(x, x)$.

		\item The sequence $(x_n)_{n\geq 1}$ is said to be convergent with respect to $\tau(p^+)$ (or $\tau(p^+)$-convergent) and converges to $x$, if $\lim\limits_{n\to \infty} p^+(x_n , x) = p^+(x, x)$.
		
		\item The sequence $(x_n)_{n\geq 1}$ is said to be $\tau(p)$-Cauchy sequence if
		$$\lim\limits_{n<m,n,m\to \infty} p(x_n , x_m)$$ exists and is finite.
		
		\item The sequence $(x_n)_{n\geq 1}$ is said to be $\tau(p^+)$-Cauchy sequence if
		
		$$\lim\limits_{n,m\to \infty} p^+(x_n , x_m)$$ exists and is finite.
		
		\item $(X, p,K)$ is said to be $\tau(p)$-complete if for every $\tau(p)$-Cauchy
		sequence $(x_n)_{n\geq 1} \subseteq X$, there exists $x \in X$ such that:
		
		$$ \lim\limits_{n<m,n,m\to \infty} p(x_n , x_m)= \lim\limits_{n\to \infty} p(x_n , x)=p(x,x).$$

	\item $(X, p,K)$ is said to be $p$-sequentially complete if every $\tau(p^+)$-Cauchy sequence is $\tau(p)$-convergent.

		\item $(X, p,K)$ is said to be $\tau(p)$-bicomplete or $\tau(p^+)$-complete if every $\tau(p^+)$-Cauchy
		sequence is $\tau(p^+)$-convergent, i.e.

		$$ \lim\limits_{n,m\to \infty} p^+(x_n , x_m)= \lim\limits_{n\to \infty} p^+(x_n , x)=p^+(x,x).$$

		\item $(X, p,K)$ is said to be $\tau(p)$-Smyth complete if every $\tau(p)$-Cauchy sequence converges for the topology $\tau(p^+)$.
	\end{enumerate}
	
\end{definition}

The following implications are easy to check:

$$ \tau(p)\text{-Smyth complete} \quad 
\Longrightarrow  \quad 
\tau(p)\text{-complete}  \quad 
\Longrightarrow  \quad 
p\text{-sequentially complete}.$$

\begin{remark}
	Observe that a sequence which is both $\tau(p)$-Cauchy and $\tau(p^{-1})$-Cauchy is $\tau(p^+)$-Cauchy. It is important to point out that the topologies induced by $p^+$ and $p^s=p\vee p^{-1}$ are the same. Hence by defining  accordingly what a $\tau(p^s)$-Cauchy sequence is, one could see that 
	$(X, p,K)$ is $\tau(p)$-bicomplete if and only if it is  $\tau(p^s)$-complete. Moreover, the $\tau(p)$-bicompleteness of a $K$-partial quasi-metric space $(X,p,K)$ is equivalent to the completeness of a partial $K$-metric space in the sense of Shukla (see \cite[Definition 4(iii)]{shu}).

	Also, note that the topology $\tau(p)$ is not Hausdorff in general, as the next example demonstrates.
\end{remark}

\begin{example}
	Let $X = [0,\infty)$, $a > 0$ be any constant and define $p : X \times X \to [0,\infty)$ by  
	$$ p(x, y) = \max\{x-y, 0\} + a \text{ for all } x,y \in X.$$
	Then $(X,p,K)$ is a $K$-partial quasi-pseudometric with arbitrary coefficient $K \geq  1.$ For a fixed positive real number $\eta$, define a sequence $(x_n)_{n\geq 1}$  in $X$ by $x_n = \eta$ for all $n \geq 1$. Note that, if $x^* \geq \eta$, we
	have $p(x_n , x^*) = a = p(x^*, x^*)$ therefore,  $\lim\limits_{n\to \infty} p(x_n , x^*) = p(x^*, x^*)$
for all	$x^* \geq \eta$. Thus, the limit of a $\tau(p)$-convergent sequence in $K$-partial quasi-metric need
	not be unique.
\end{example}

In the sequel, we shall define appropriate notions of contractive maps, suitable for the category of $K$-partial quasi-metrics. We recall the following results, due to Shukla\cite{shu} and that will be useful.

\begin{theorem}\label{shutheorem1}(\cite[Theorem 1.]{shu})
	Let $(X, b)$ be a complete partial $b$-metric space with coefficient
	$s \geq 1$ and $T : X \to X$ be a mapping satisfying the following condition: 
	
	\[b(T x, T y) \leq \lambda b(x, y)\quad 
	\text{ for all } x, y \in X,\]
	
	where $\lambda \in [0, 1)$. Then $T$ has a unique fixed point $u \in X$ and $b(u, u) = 0.$
\end{theorem}

\begin{theorem}\label{shutheorem2}(\cite[Theorem 2.]{shu})
	Let $(X, b)$ be a complete partial $b$-metric space with coefficient
	$s \geq 1$ and $T : X \to X$ be a mapping satisfying the following condition: 
	
	\[b(T x, T y) \leq \lambda [b(x, T x) + b(y, T y)]\quad 
	\text{ for all } x, y \in X,\]
	
	where $\lambda \in [0, 1/2), \lambda \neq 1/s$. Then $T$ has a unique fixed point $u \in X$ and $b(u, u) = 0.$
\end{theorem}

\begin{definition}
	Let $(X, p,K)$ be a $K$-partial quasi-metric space.
	
	\begin{enumerate}
	\item By a $p$-Banach mapping on $X$, we mean a self-mapping $T$ on $X$ such that there exists a constant $0\leq \lambda<1$ satisfying
	$$p(T x, T y) \leq \lambda p(x, y)
	\quad \text{ for all } x, y \in X.$$
	
	\item By a $p$-Kannan mapping on $X$, we mean a self-mapping $T$ on $X$ such that there exists a constant $,0\leq \lambda K < 1/2$ satisfying
	$$p(T x, T y) \leq \lambda [p(Tx, x) + p(y,Ty)]
	\quad \text{ for all } x, y \in X.$$
	
	\end{enumerate}
	
\end{definition}

Now, we can state the following theorem, analogue to Banach contraction principle in $K$-partial quasi-metric space. We begin with this lemma:

\begin{lemma}\label{lemma1}
	Let $T$ be a $p$-Banach mapping on $K$-partial quasi-metric space $(X, p,K)$ with $0\leq \lambda<1$. Then:
	
	\begin{enumerate}

		\item[(a)] $T$ is a Banach mapping on the partial $K$-metric space $(X,p^+)$ with $0\leq \lambda<1$, i.e.
		$$p^+(Tx,Ty) \leq \lambda p^+(x,y), \quad \text{ for all } x, y \in X.$$
		
		\item[(b)] For any $x_0 \in X$, the sequence $(T^n x_0)_{n\geq n_0}$, for some $n_0\in \mathbb{N}$, is $\tau(p^+)$-Cauchy sequence.
	\end{enumerate} 
\end{lemma}

\begin{proof}
	
(a) Given $x,y \in X$, we have

$$p(T x, T y) \leq \lambda p(x, y) \quad \text{and} \quad p(T y, T x) \leq \lambda p(y, x),$$

so 

$$p^+(T x, T y) \leq \lambda p^+(x, y). $$
It follows that $T$ is a Banach mapping on the partial $K$-metric space $(X,p^+)$ with $0\leq \lambda<1$.

\vspace*{0.3cm}

(b) Since $T$ is a Banach mapping on the partial $K$-metric space $(X,p^+)$, the classical proof of Theorem \ref{shutheorem1} shows that for any $x_0 \in X,$ the sequence $(T^n x_0)_{n\geq n_0}$, for some $n_0\in \mathbb{N}$, is $\tau(p^+)$-Cauchy sequence.
\end{proof}

A reasonable and straightforward formulation of the Banach contraction principle in the setting of a $K$-partial quasi-metric space seems to be:

\begin{theorem}\label{p-banach}
Let $(X, p,K)$ be a $\tau(p)$-bicomplete $K$-partial quasi-metric space. Then every $p$-Banach mapping on $(X, p,K)$ has a unique fixed point $x^*$ and $p(x^*,x^*)=0.$
\end{theorem} 

\begin{proof}
	From Lemma \ref{lemma1}, we know $T$ is a Banach mapping on the partial $K$-metric space $(X,p^+)$ with $0\leq \lambda<1$. The classical proof of Theorem \ref{shutheorem1} then shows that $T$ has a unique fixed point $x^*\in X$ and $p(x^*,x^*)=0.$
\end{proof}

However, in view of having minimal condition, one can state the following refined version of Theorem \ref{p-banach}.

\begin{theorem}\label{p-banach1}
	Let $(X, p,K)$ be a $\tau(p)$-Smyth complete $K$-partial quasi-metric space. Then every $p$-Banach mapping on $(X, p,K)$ has a unique fixed point $x^*$ and $p(x^*,x^*)=0.$
\end{theorem} 

\begin{proof}
	From the proof of Theorem \ref{p-banach}, for any initial point $x_0\in X$, the sequence of iterates $\{ T^nx_0\}$ is $\tau(p)$-Cauchy. The conclusion follows immediately.
\end{proof}

Our aim, in setting up this theory, is also to minimize the completeness assumptions on the space $(X, p,K)$ and still guarantee the existence of a unique fixed point. However, we have the following problem:

\begin{problem}\label{problem1}
	We would like to give a counter-example to the following statement:
	\textcolor{blue}{Let $(X, p,K)$ be a $p$-sequentially complete $K$-partial quasi-metric space. Then every $p$-Banach mapping on $(X, p,K)$ has a unique fixed point $x^*$ and $p(x^*,x^*)=0.$}
\end{problem}

If we want to keep the $p$-sequential completeness of the  $K$-partial quasi-metric space $(X, p,K)$, we need to require that $(X, p,K)$ be Hausdorff.
 
So we have 

\begin{theorem}\label{p-banach2}
	Let $(X, p,K)$ be a Hausdorff $p$-sequentially complete $K$-partial quasi-metric space. Then every $p$-Banach mapping on $(X, p,K)$ has a unique fixed point $x^*$ and $p(x^*,x^*)=0.$
\end{theorem}

\begin{proof}
	From the proof of Theorem \ref{p-banach}, for any initial point $x_0\in X$, the sequence of iterates $\{ T^nx_0\}$ is $\tau(p)$-convergent to some $u\in X$ such that $p(u,u)=0$.
Hence $p(Tu,Tu)\leq \lambda p(u,u)$ implies that $p(Tu,Tu)=0$. Moreover 
\[ \lim\limits_{n\to \infty} p(Tx_n,Tu) \leq \lambda p(x_n,u) = p(u,u)=0 \]
implies that \[\lim\limits_{n\to \infty}p(Tx_n,Tu) = 0 = p(Tu,Tu) =\lim\limits_{n\to \infty}p(x_{n+1},Tu).  \]	

That is \[ \lim\limits_{n\to \infty}p(x_{n+1},Tu) =p(Tu,Tu)=p(u,u)=0, \]	
and by virtue of $(X, p,K)$ being Hausdorff, $Tu=u.$
This conclude the proof.
\end{proof}

Next, we look at the existence result for a Kannan type mappings in $K$-partial quasi-metric spaces. Here again, the results by Shukla's \cite{shu} will be of great use.

\begin{lemma}\label{lemma2}
	Let $T$ be a $p$-Kannan mapping on the $K$-partial quasi-metric space $(X, p,K)$ with $0\leq \lambda<1/2$  and $\lambda\neq 1/K$. Then:

	\begin{enumerate}

		\item[(a)] $T$ is a Kannan mapping on the partial $K$-metric space $(X,p^+)$ with $0\leq \lambda<1/2, \lambda < 1/K$, i.e.
		$$p^+(Tx,Ty) \leq \lambda [p^+(x,Tx) + p^+(y,Ty)], \quad \text{ for all } x, y \in X.$$
		
		\item[(b)] For any $x_0 \in X$, the sequence $(T^n x_0)_{n\geq 1}$, is $\tau(p^+)$-Cauchy sequence.
	\end{enumerate} 
\end{lemma}

\begin{proof}

(a) Given $x,y \in X$, we have

$$p(T x, T y) \leq \lambda [p(Tx, x) +p(y,Ty)] \quad \text{and} \quad p(T y, T x) \leq \lambda [p(Ty,y) + p(x, Tx)],$$

so 

$$p^+(T x, T y) \leq \lambda [p^+(x, Tx)+ p^+(y,Ty)]. $$
It follows that $T$ is a Kannan mapping on the partial $K$-metric space $(X,p^+)$ with $0\leq \lambda<1/2$ and $\lambda< 1/K$.

\vspace*{0.3cm}

(b) Since $T$ is a Kannan mapping on the partial $K$-metric space $(X,p^+)$, the classical proof of Theorem \ref{shutheorem2} shows that for any $x_0 \in X,$, the sequence $(T^n x_0)_{n\geq 1}$, is $\tau(p^+)$-Cauchy sequence.
\end{proof}

Here again, we are facing the following problem:

\begin{problem}\label{problem2}
	We would like to give a counter-example to the following statement:
	\textcolor{blue}{Let $(X, p,K)$ be a $p$-sequentially complete $K$-partial quasi-metric space. Then every $p$-Kannan mapping on $(X, p,K)$ has a unique fixed point $x^*$ and $p(x^*,x^*)=0.$}
\end{problem}

So, a reasonable and straightforward formulation of the Kannan contraction principle in the setting of a $K$-partial quasi-metric space seems to be:

\begin{theorem}\label{p-kannan}
	Let $(X, p,K)$ be a $\tau(p)$-bicomplete $K$-partial quasi-metric space. Then every $p$-Kannan mapping on $(X, p,K)$ with constant $0<\lambda<1/2$ and $\lambda\neq 1/K$ has a unique fixed point $x^*$ and $p(x^*,x^*)=0.$
\end{theorem} 

\begin{proof}
	From Lemma \ref{lemma2}, we know $T$ is a Kannan mapping on the partial $K$-metric space $(X,p^+)$ with $0\leq \lambda<1$. The proof of Theorem\ref{shutheorem2} then shows that $T$ has a unique fixed point $x^*\in X$ and $p(x^*,x^*)=0.$
\end{proof}

The formulation using $\tau(p)$-Smyth completeness works as well.

\begin{theorem}\label{p-kannan}
	Let $(X, p,K)$ be a $\tau(p)$-Smyth complete $K$-partial quasi-metric space. Then every $p$-Kannan mapping on $(X, p,K)$ with constant $0<\lambda<1/2$ and $\lambda\neq 1/K$ has a unique fixed point $x^*$ and $p(x^*,x^*)=0.$
\end{theorem}

In the coming lines, we obtain some Reich type fixed point theorems in $K$-partial quasi-metric spaces by combining the Banach and the Kannan contraction conditions.

\begin{definition}\label{p-reich}
	Let $(X, p,K)$ be a $K$-partial quasi-metric space.
	
 By a $p$-Reich mapping on $X$, we mean a self-mapping $T$ on $X$ such that there exist nonnegative constants $ \lambda, \mu, \delta$ satisfying $\lambda+ \mu+ \delta<1/K$ and

		$$p(T x, T y) \leq \lambda p(x, y) + \mu p(Tx, x) + \delta p(y,Ty)
		\quad \text{ for all } x, y \in X.$$
	
	We shall say that $T$ is a $p$-Reich mapping with constants $ \lambda, \mu, \delta$.
\end{definition}

\begin{lemma}\label{lemma3f}
	Let $T$ be a $p$-Reich mapping with constants $ \lambda, \mu, 2\delta$, on the $K$-partial quasi-metric space $(X, p,K)$. Then

	\begin{enumerate}

		\item[(a)] $T$ is a Reich mapping on the partial $K$-metric space $(X,p^+)$, i.e. there exist some nonnegative constants $ a,b,c$ satisfying $a+ b+ c<1/K$ and 
		$$p^+(Tx,Ty) \leq a p^+(x, y) + b p^+(Tx, x) + c p^+(y,Ty), \quad \text{ for all } x, y \in X.$$
		
		\item[(b)] For any $x_0 \in X$, the sequence $(T^n x_0)_{n\geq 1}$, is a $\tau(p^+)$-Cauchy sequence.
	\end{enumerate} 
\end{lemma}

\begin{proof}
Analogue to the proofs of Lemma \ref{lemma1} and Lemma \ref{lemma2}.

\end{proof}

We conclude this section by stating, without proving (since the proof is straightforward), analogue of Theorems \ref{p-banach} and \ref{p-kannan}

\begin{theorem}
	Let $(X, p,K)$ be a $\tau(p)$-bicomplete $K$-partial quasi-metric space. Then every $p$-Reich mapping on $(X, p,K)$ has a unique fixed point $x^*$ and $p(x^*,x^*)=0.$
\end{theorem}

\begin{remark}
	
	One obtains a similar result it one replaces the condition
	$$p(T x, T y) \leq \lambda p(x, y) + \mu p(Tx, x) + \delta p(y,Ty)
	\quad \text{ for all } x, y \in X$$
	
from Definition \ref{p-reich} by

	\begin{equation}\label{maxcond}
	p(T x, T y) \leq \lambda \max\{p(x, y),p(Tx, x),p(y,Ty)\}
	\quad \text{ for all } x, y \in X,
	\end{equation}

where $\lambda$ is such that $3\lambda<1.$

\end{remark}

\section{$p$-Chatterjea Contractions }

We continue our development by providing a Chatterjea type fixed point results in $K$-partial quasi-metric space. We fist give the result in the setting of partial $K$-metric spaces and then extend it. We begin with the definition of a $p$-Chatterjea contraction:

\begin{definition}
	Let $(X, p,K)$ be a $K$-partial quasi-metric space. By a $p$-Chatterjea mapping on $X$, we mean a self-mapping $T$ on $X$ such that there exists a constant $0\leq \lambda K \leq \lambda K^2 < 1/2$ satisfying
		$$p(T x, T y) \leq \lambda [p(x, Ty) + p(Tx,y)]
		\quad \text{ for all } x, y \in X.$$
	\end{definition}

The following theorem is an analogue to Chatterjea fixed point theorem in partial $K$-metric space.

\begin{theorem}\label{thmchater}
	Let $(X,p ,K)$ be a complete partial $K$-metric space with coefficient
	$K\geq  1$ and $T : X \to X$ be a mapping satisfying the following condition:
	\begin{equation}\label{chater}
	p(T x, T y) \leq \lambda [p(x, Ty) + p(Tx,y)]
	\quad \text{ for all } x, y \in X,
	\end{equation}
	with $0\leq \lambda K \leq \lambda K^2 < 1/2$. Then $T$ has a unique fixed point $x^*\in X$ and $p(x^*,x^*)=0.$
	
\end{theorem}

\begin{proof}\hspace*{0.2cm}
	
	\vspace*{0.3cm}
	
	\underline{Uniqueness:}
	Let us first show that if $T$ has a fixed point, then it is unique. We shall
	show that, if $u \in X$ is a fixed point of $T$, that is, $T u = u$, then $p(u, u) = 0$.
	From \eqref{chater} we obtain
	$$p(u, u) = p(T u, T u) \leq \lambda[p(u, T u) + p(Tu,  u)] = 2\lambda p(u, u) < p(u, u),$$
	a contradiction. Therefore the equality $p(u, u) = 0$ must hold.
	At this point, it is crucial to recall that in a partial $K$-metric space $(X,p ,K)$, if $x, y \in X$ and $p(x, y) = 0$, then
	$x = y$.
	Suppose now that $u, v \in X$ are two fixed points of $T$, that is, $T u = u, T v = v$. From \eqref{chater}, we can write
	\begin{align*}
			p(u, v) = p(T u, T v) &\leq \lambda [p(u, T u) + p(Tv, v)] \\
			& = \lambda [p(u, u) + p(v, v)] = 0.
	\end{align*}
Therefore, we must have $p(u, v) = 0$, that is, $u = v$. Thus if a fixed point of $T$ exists, then it is unique.

\vspace*{0.3cm}

\underline{Existence:} For existence of fixed point, let $x_0 \in X$ be arbitrary and set $x_n = T^n x_0$ and $p_n = p(x_n , x_{n+1} )$. We can assume, without loss of generality that $p_n >0$ for all $n \geq 0$, otherwise $x_n$ is a fixed point of $T$ for at least one $n_0 \geq 0$.

For any $n\geq 0$, it follows from \eqref{chater} that

\begin{align*}
	 p_n &= p(x_n , x_{n+1} ) = p(T x_{n-1} , T x_n ) \\
	     & \leq \lambda [p(x_{n-1} , T x_n ) + p(T x_{n-1} , x_n )] \\
	     & = \lambda [p(x_{n-1} , x_{n+1} ) + p(x_{n} , x_n )] \\
	     &\leq \lambda K [p(x_{n-1} , x_{n} ) + p(x_{n} , x_{n+1} )] \\
	     & \leq \lambda K [p_{n-1} + p_n],
\end{align*}

therefore, $p_n \leq \mu p_{n-1}$ , where $\mu = \frac{\lambda K}{1- \lambda K } < 1$ (as $\lambda < 1/2K$ ). On repeating this process, we obtain

$$p_n \leq \mu^n p_0.$$ Therefore, $\lim\limits_{n\to \infty} p_n = 0$.

Now, for $n, m \in \mathbb{N}$, we have

\begin{align*}
	p(x_n , x_m ) &= p(T^n x_0 , T^m x_0 ) = p(Tx_ {n-1} , T x_{m-1} ) \\
	& \leq \lambda [ p(x_ {n-1},x_ {m}) + p(x_ {n},x_ {m-1})].
\end{align*}

On the one side, we have 

\[ p(x_ {n-1},x_ {m}) \leq  K(p(x_{n-1},x_n)+p(x_n,x_m)) -p(x_n,x_n)\]

and on the other 

\[ p(x_ {n},x_ {m-1}) \leq  K(p(x_n,x_m) + p(x_m,x_{m-1})) -p(x_m,x_m) ,\]

which yields 
\[ p(x_n,x_m) \leq -\frac{\lambda(p(x_n,x_n)+ p(x_m,x_m)) + \lambda K (p_{n-1} + p_{m-1})}{1-2\lambda K}. \]

Also, we note that if we apply the condition \eqref{chater} to the couple $(x_n,x_n)$, we have

\begin{align*}
	p(x_n,x_n) & = p(Tx_{n-1},Tx_{n-1})\\
	           & \leq \lambda [ p(x_{n-1},x_{n}) + p(x_{n},x_{n-1})] \\
	           &\leq  2 \lambda p(x_{n-1},x_{n}) = 2 \lambda p_n.
\end{align*}

As $\lim\limits_{n\to \infty} p_n = 0$, $\lim\limits_{n\to \infty} 	p(x_n,x_n)= 0$ and $\lim\limits_{n,m\to \infty}p(x_n,x_m) =0 $. Thus, the sequence $(x_n)_{n\geq 1}$ is a Cauchy sequence in $X$.

By completeness of $X$ there exists $u \in X$ such that

\begin{equation}\label{complete}
\lim\limits_{n\to \infty} p(x_n,u) = \lim\limits_{n,m\to \infty}p(x_n,x_m) = p(u,u)= 0.
\end{equation}

Now, we show that $u$ is a fixed point of $T$. 
Using again \eqref{chater}, we write

\begin{align*}
p(u, T u) & \leq	K[p(u, x_{n+1} ) + p(x_{n+1} , Tu)] -p(x_{n+1} , x_{n+1} ) \\
& \leq K[p(u, x_{n+1} ) + p(x_{n+1} , Tu)]\\
& \leq K[p(u, x_{n+1} ) + p(Tx_{n} , Tu)]\\
& \leq K[p(u, x_{n+1} ) +  \lambda (p(x_{n+1},u) + p(x_n,Tu )]\\
& \leq K [ p(u, x_{n+1} )+\lambda (p(x_{n+1},u) + \lambda (    K[ p(x_n,u)+p(u,Tu)-p(u,u)]            )           ] \\
& \leq K [ p(u, x_{n+1} ) + \lambda (    K[ p(x_n,u)+p(u,Tu)]            )           ] \\
&  = K p(u, x_{n+1} ) + \lambda K^2 p(x_n,u)+ \lambda K^2 p(u,Tu),
\end{align*}

which yields 

\[  p(u,Tu) \leq \frac{K(\lambda+1)p(u,x_{n+1})+ \lambda K^2 p(x_n,u)}{1-\lambda K^2}.\]

Note that $\lambda \neq 1/K^2$, so by \eqref{complete}, we conclude that $p(u, T u) = 0$, i.e. $T u = u$. Thus, $u$ is a unique fixed point of $T$.

\end{proof}

Now, we formulate Theorem \ref{thmchater} in the asymmetric setting by:

\begin{theorem}\label{thm2}
	Let $(X,p ,K)$ be a $\tau(p)$-bicomplete partial $K$-metric with $K>1$. Then every $p$-Chatterjea contraction 
	with coefficient $0<\lambda K < \lambda K^2<1/2$ has a unique fixed point $x^*\in X$ and $p(x^*,x^*)=0.$
	
\end{theorem}

In view of the proof of this theorem, we shall make use of the following lemma:

\begin{lemma}\label{lemma3}
	
Let $T$ be a $p$-Chatterjea mapping on the $K$-partial quasi-metric space $(X, p,K)$ with $0<\lambda K \leq \lambda K^2<1/2$. Then:

\begin{enumerate}

	\item[(a)] $T$ is a Chatterjea mapping on the partial $K$-metric space $(X,p^+)$ with $0\leq \lambda K<1/2$ and $\lambda\neq 1/K^2$, i.e.
	$$p^+(Tx,Ty) \leq \lambda [p^+(x,Ty) + p^+(Tx,y)], \quad \text{ for all } x, y \in X.$$
	
	\item[(b)] For any $x_0 \in X$, the sequence $(T^n x_0)_{n\geq 1}$, is $\tau(p^+)$-Cauchy sequence.
\end{enumerate}
\end{lemma}

\begin{proof}
	
	(a) Given $x,y \in X$, we have
	
	$$p(T x, T y) \leq \lambda [p(x, Ty) +p(Tx,y)] \quad \text{and} \quad p(T y, T x) \leq \lambda [p(y,Tx) + p(Ty, x)],$$
	
	so 
	
	$$p^+(T x, T y) \leq \lambda [p^+(x, Ty)+ p^+(Tx,y)]. $$
	It follows that $T$ is a Chatterjea mapping on the partial $K$-metric space $(X,p^+)$ with $0\leq \lambda K<1/2$ nd $\lambda\neq 1/K^2$.
	
	(b) Since $T$ is a Chatterjea mapping, i.e. a mapping that satisfies \eqref{chater} on the partial $K$-metric space $(X,p^+)$, the proof of Theorem \ref{thmchater} shows that for any $x_0 \in X,$, the sequence $(T^n x_0)_{n\geq 1}$, is $\tau(p^+)$-Cauchy sequence.
\end{proof}

Now we present the proof to Theorem \ref{thm2}.
\begin{proof}
	From Lemma \ref{lemma3}, we know $T$ is a Chatterjea mapping on the partial $K$-metric space $(X,p^+)$ with $0\leq \lambda K<1/2$ and $\lambda\neq 1/K^2$. The proof of Theorem \ref{thmchater} then shows that $T$ has a unique fixed point $x^*\in X$ and $p(x^*,x^*)=0.$
\end{proof}

Our next step is to find a way to define a more general class of contractions which includes the three already mentioned in this manuscript.
We then introduce the concept of weak contraction for self mappings defined on $K$-partial quasi-metric spaces. The main merit of weak contractions, as already observed, in the metrical contractive type mappings is that they
unify large classes of contractive type operators, whose fixed points can be obtained by means of the Picard iteration.

\begin{definition}
	Let $(X, p,K)$ be a $K$-partial metric space. A map $T : X \to X$ is
	called weak contraction if there exist a constant $\delta \in (0, 1)$ and some
	$L \geq 0$ such that
	
	\begin{equation}\label{weakcontraction1}
		p(T x, T y) \leq \delta · p(x, y) + Lp(Tx, y) , \quad 
		\text{for all} x, y \in X .
	\end{equation}
\end{definition}

\begin{remark}
Due to the symmetry of the $p$-distance, the weak contraction
condition \eqref{weakcontraction1} implicitly includes the following dual one, namely
\begin{equation}\label{weakcontraction2}
p(T x, T y) \leq \delta · p(x, y) + Lp(x, Ty) , \quad 
\text{for all} x, y \in X .
\end{equation}
Consequently, in order to check the weak contractiveness of $T$, it
is necessary to check both \eqref{weakcontraction1} and \eqref{weakcontraction1}. It is then obvious that by setting 
$L = 0$, we recover the Banach principle for $K$-partial metric space (see \cite[Theorem 1.]{shu}) 
and hence the Banach principle is a weak contraction (that possesses a unique fixed
point).
\end{remark}

Other examples of weak contractions are given by the next propositions.

\begin{proposition}\label{prop1}
	Let $(X, p,K)$ be a $K$-partial metric space.
	Any Kannan mapping, i.e. any mapping satisfying
	the contractive condition:
	
	\begin{equation}\label{weakkannan}
		p(T x, T y) \leq \lambda [p(x, Tx)+ p(y,Ty)],
			\end{equation}
whenever $x,y \in X$ with $0\leq \lambda K<1/2$, is a weak contraction.
\end{proposition}

\begin{proof}
	
	\begin{align*}
	p(T x, T y) & \leq \lambda [p(x, T x) + p(y, T y) ] \\
	\end{align*}
One the one hand, we have 
$$ p(x, T x) \leq K[p(x,y) + p(y,Tx)]-p(x,x) \leq K[p(x,y) + p(y,Tx)],$$
and the other hand

$$p(y,Ty)\leq K[p(y,Tx)+p(Tx,Ty)] - p(Tx,Tx) \leq K[p(y,Tx)+p(Tx,Ty)].$$

So 

$$p(T x, T y) \leq \lambda(  K[p(x,y) + p(y,Tx)]+ K[p(y,Tx)+p(Tx,Ty)] ), $$

which yields
$$ p(T x, T y) \leq  \frac{\lambda K}{1-\lambda K}\ p(x,y) + \frac{2\lambda K}{1-\lambda K}\ p(Tx,y) \quad \text{for all } x, y \in X,$$

i.e., in view of $0< \lambda K \leq 1/2$,  \eqref{weakcontraction1} holds with $\delta = \frac{\lambda K}{1-\lambda K}$ and $L = \frac{2\lambda K}{1-\lambda K}.$
Since \eqref{weakkannan} is symmetric with respect to $x$ and $y$, \eqref{weakcontraction2} also holds.
\end{proof}

\begin{proposition}\label{prop2}
	Let $(X, p,K)$ be a $K$-partial metric space.
	Any Chatterjea mapping, i.e. any mapping satisfying
	the contractive condition:
	
	\begin{equation}\label{weakchatterjea}
	p(T x, T y) \leq \lambda [p(x, Ty)+ p(Tx,y)],
	\end{equation}
	whenever $x,y \in X$ with $0\leq \lambda K < \lambda K^2<1/2$, is a weak contraction.
\end{proposition}

\begin{proof}

We note that 
$$ p(x, T y) \leq K[p(x, y) + p(y,Ty)] -p(y,y) \leq K[p(x, y) + p(y,Ty)],$$ 
and also that

$$p(y, T y) \leq K [p(y,Tx) + p(Tx,Ty)] - p(Tx,Tx) \leq  K [p(y,Tx) + p(Tx,Ty)] .$$

So 

$$p(Tx,Ty) \leq \lambda [Kp(x,y)+ K^2p(y,Tx)+K^2p(Tx,Ty) +p(Tx,y)],$$

which yields

$$ p(Tx,Ty) \leq \frac{\lambda K}{1-\lambda K^2}\ p(x,y) + \frac{1+\lambda K^2}{1-\lambda K^2}\ p(Tx,y),$$

which is \eqref{weakcontraction1} with $\delta = \frac{\lambda K}{1-\lambda K^2} <1$ (since $0\leq \lambda K^2<1/2$) and $L = \frac{1+\lambda K^2}{1-\lambda K^2}\geq 0$ (since $\lambda< 1/K^2$ ).
The symmetry of \eqref{weakchatterjea} also implies \eqref{weakcontraction2}.
\end{proof}

One of the most general contraction condition, also discussed by Ili\'c et al. \cite{ilic}, and for which the map
satisfying it is still a Picard operator, is the so-called quasi contraction and has initially  been obtained by Ciric\cite{cir} in
1974: there exists $0 < h < 1$ such that

\begin{equation}\label{zame}
	d(T x, T y) \leq h · \max \{d(x, y), d(x, T x), d(y, T y), d(x, T y), d(y, T x)\}, \quad \text{for all } x,y \in X ,
\end{equation}
where $(X,d)$ is a metric space.

\vspace*{0.3cm}

 Our main aim in the coming lines is to express quasi contractions as weak contractions, in the setting of a partial $K$-partial metric space.

\begin{proposition}
	Let $(X, p,K)$ be a $K$-partial metric space.
	Any quasi contraction, i.e. any mapping satisfying
	
	\begin{equation}\label{weakzame}
	p(T x, T y) \leq h · \max \{p(x, y), p(x, T x), p(y, T y), p(x, T y), p(y, T x)\}, \quad \text{for all } x,y \in X ,
	\end{equation}
	
	 with $0 < h K < hK^2< 1/2$ is a weak
	contraction.
\end{proposition}

\begin{proof}
Set $M(x,y) = \max \{p(x, y), p(x, T x), p(y, T y), p(x, T y), p(y, T x)\}.$ 

Let $x, y \in X$ be arbitrary taken. We have to discuss five possible cases.

\vspace*{0.3cm}

\begin{enumerate}
\item[\textbf{Case 1.}] $M (x, y) = p(x, y)$, when, in virtue of \eqref{weakzame}, condition \eqref{weakcontraction1} and \eqref{weakcontraction2} are obviously satisfied (with $\delta= h$ and $L = 0$).
\end{enumerate}

Since $M (x, y) = M (y, x)$, for the four remaining cases, it suffices to prove that at least one of the relations \eqref{weakcontraction1} or \eqref{weakcontraction2} holds. (We sometimes however prove the both inequalities).

\vspace*{0.3cm}

\item[\textbf{Case 2.}]  $M (x, y) = p(x, T x)$, by \eqref{weakzame} and triangle rule

$$p(T x, T y) \leq h p(x, T x) \leq h K [p(x, y) + p(y, T x)]- h p(y,y) \leq h K [p(x, y) + p(y, T x)] ,$$

and so \eqref{weakcontraction1} holds with $\delta = hK$ and $L = hK$.

Again, by triangle rule $$p(x, T x) \leq K[p(x, T y) + p(T y, T x)] - p(Ty,Ty) \leq  K[p(x, T y) + p(T y, T x)],$$ hence

$$p(T x, T y) \leq h  K[p(x, T y) + p(T y, T x)].$$
Therefore,

$$p(T x, T y) \leq \frac{hk}{1-hk}p(x,Ty) \leq \delta p(x,y)+\frac{hk}{1-hk}p(x,Ty) $$

for any $\delta\in (0,1).$ So \eqref{weakcontraction2} also holds.

\vspace*{0.3cm}

\item[\textbf{Case 3.}] $M (x, y) = p(y, T y)$, when \eqref{weakcontraction1} and \eqref{weakcontraction2} follow by Case 2, in
virtue of the symmetry of $M (x, y).$

\vspace*{0.3cm}

\item[\textbf{Case 4.}] $M (x, y) = p(x, T y)$, when \eqref{weakcontraction2} is obviously true and \eqref{weakcontraction1} is obtained only if $hK<\frac{1}{2}$ and $\lambda< 1/K^2$. Indeed, using triangle inequality, we have

$$p(x, T y) \leq K[ p(x, y) + p(y, T y)],$$

and 

$$ p(y,Ty) \leq K[p(y,Tx) + p(Tx,Ty)],$$

i.e.

$$ p(x,Ty) \leq K p(x,y) + K^2p(y,Tx) + K^2 p(Tx,Ty).$$

Then one obtains
$$ p(Tx,Ty) \leq hp(x,Ty)\leq  hK p(x,y) + hK^2p(y,Tx) + hK^2 p(Tx,Ty),$$

i.e. 

$$p(Tx,Ty) \leq \frac{hK}{1-hK^2} p(x,y) + \frac{hK^2}{1-hK^2} p(y,Tx),$$

which is \eqref{weakcontraction1} with $\delta = \frac{hK}{1-hK^2} <1$ (since $hK^2<1/2$) and $L=\frac{hK^2}{1-hK^2}>0.$

\vspace*{0.3cm}

\item[\textbf{Case 5.}] $M (x, y) = p(y, T x)$, which reduces to Case 4.
The proof is complete.

\end{proof}

\begin{remark}
It is easy to see that 

\begin{align*}
	p(Tx,T^2x) & \leq \delta p(x,Tx) + L p(Tx,Tx)\\
	           & \leq \delta p(x,Tx)+LK (p(Tx,x)+p(x,Tx)-p(x,x))\\
	           & \leq (\delta + 2LK)p(x,Tx),
\end{align*}

i.e. condition \eqref{weakcontraction1} implies the so called Banach
orbital condition, studied by various authors in the context of fixed point theorems.
\end{remark}

The main result of this section is given by

\begin{theorem}\label{main1}
	Let $(X, p,K)$ be a complete partial $K$-metric space and $T : X \to X$
	a weak contraction, i.e. a mapping satisfying \eqref{weakcontraction1} with $\delta  \in (0, 1)$ and
	some $L \geq 0$ such that $\delta+2L<1/K$. Then

	\begin{enumerate}
	\item[(a)] $F (T ) = \{x \in X : T x = x\} \neq \emptyset$ and whenever $x\in F(T), p(x,x)=0$;
	\item[(b)] For any $x_0 \in X$, the Picard iteration given by $(T^nx_0)_{n\geq 0}$ converges to some $x^* \in F (T )$;
	\item[(c)] The following estimates
	
	\begin{equation}\label{estimate1}
	 p(x_m,x^*) \leq \frac{K\lambda^m}{1-K\lambda}p(x_0,x_1), \qquad m=0,1,2,\cdots
	\end{equation}	
	
	\begin{equation}\label{estimate2}
	p(x_m,x^*) \leq \frac{K\lambda}{1-K\lambda}p(x_{m-1},x_m), \qquad m=1,2,\cdots
	\end{equation}
hold.	
	
	\end{enumerate}
\end{theorem}

\begin{proof}
	We shall prove that $T$ has at least a fixed point in $X$. Let $x_0 \in X$ be arbitrary; set $x_n = T^n x_0$ and $p_n = p(x_n , x_{n+1} )$.
	Take $x := x_{n-1}, y := x_n$ in \eqref{weakcontraction1} to obtain
	\begin{align*}
		p(Tx_{n-1},Tx_n) & \leq \delta p(x_{n-1},x_n) + L  p(x_n,x_n) \\
		& \leq \delta p(x_{n-1},x_n) + K[p(x_{n},x_{n+1}) + p(x_{n+1},x_{n})] - p(x_{n+1},x_{n+1})\\
		& \leq \delta p(x_{n-1},x_n) + 2LK p(x_{n},x_{n+1}),
	\end{align*}
	
	which shows that
	
	\begin{equation}\label{induc}
		p(x_n, x_{n+1})  \leq \frac{\delta}{1-2LK} p(x_{n-1},x_n).
	\end{equation}

Using \eqref{induc} we obtain by induction

$$ p_{n} \leq \lambda ^n p_0, \quad n = 0, 1, 2, \cdots ,$$	

where $\lambda = \frac{\delta}{1-2LK} <1$ since $\delta+2LK<1.$

Therefore, $\lim\limits_{n\to \infty}p_n = 0$. Now we shall show that $(T^nx_0)_{n\geq 0}$ is a Cauchy
sequence. It follows from \eqref{weakcontraction1} that for $n, m \in \mathbb{N}$ 

\begin{align*}
	p(x_n , x_m ) & = p(T^n x_0 , T^m x_0 ) = p(Tx_{n-1} , Tx_{ m-1} )\\
	& \leq \delta p(x_{n-1} , x_{ m-1}) + L p(x_{n} , x_{ m-1}).
\end{align*}	

Applying the triangle rule to $p(x_{n-1} , x_{ m-1})$ via $x_n$ and to $p(x_{n} , x_{ m-1})$ via $x_m$, we obtain, after simplifications

\begin{equation}\label{esti1}
	p(x_n,x_m) \leq \frac{1}{1-LK-K^2}\ [Kp_{n-1} + K^2p_{m-1}+LKp_{m-1}].
\end{equation}

As $\lim\limits_{n\to \infty} p_n = 0$, $\lim\limits_{n\to \infty} p(x_n,x_m) = 0$ and $(T^nx_0)_{n\geq 0}$ is a Cauchy
sequence.

By completeness of $X$, there exists $u \in X$ such that

\begin{equation}\label{complete1}
\lim\limits_{n\to \infty} p(x_n,u) = \lim\limits_{n,m\to \infty}p(x_n,x_m) = p(u,u)= 0.
\end{equation}

We show that $u\in F(T)$. Indeed

\begin{align*}
	p(u, T u)& \leq K[p(u, x_{n+1} ) + p(x_{n+1} , T u)]  -p(x_{n+1},x_{n+1}) \\
	        & \leq K[p(u, x_{n+1} ) + p(x_{n+1} , T u)] \\
	        & \leq K[p(u, x_{n+1} ) + \{ \delta p(x_n,u)+ L p(x_{n+1},u)\} ].
\end{align*}

Therefore, it follows from \eqref{complete1} and the above inequality that
$p(u, T u) = 0$, that is, $T u = u$ and $F(T) \neq \emptyset$.

For $m, n=m+l$

\begin{align}\label{hence}
p(x_m,x_{m+l}) & \leq K[p(x_m,x_{m+1}) + p(x_{m+1},x_{m+l})]-p(x_{m+1},x_{m+1}) \nonumber  \\
& \leq K[p(x_m,x_{m+1}) + p(x_{m+1},x_n)] \nonumber \\
& \leq K p(x_m,x_{m+1}) + K^2 p(x_{m+1},x_{m+2}) +K^3 [p(x_{m+2},x_{m+3})+ p(x_{m+3},x_{m+l})] \nonumber \\
&  \vdots \nonumber \\
& \leq K \lambda^m(1+K\lambda + K^2\lambda^2+\cdots + K^{l-1}\lambda^{l-1})p(x_0,x_1) \nonumber \\
& = \frac{K\lambda^m}{1-K\lambda}(1-(K\lambda)^l)p(x_0,x_1).
\end{align}

Letting $l\to \infty, $ we obtain 

\[ p(x_m,x^*) \leq \frac{K\lambda^m}{1-K\lambda}p(x_0,x_1). \]

Moreover, by \eqref{induc} we inductively obtain

\[ p(x_{m+k},x_{m+k+1})  \leq \lambda^{k+1} p(x_{m-1},x_m) \qquad k,n \in \mathbb{N} \]

and hence, similarly to deriving \eqref{hence} we obtain

\begin{equation}\label{hence2}
p(x_m,x_{m+l})  \leq \frac{K\lambda}{1-K\lambda}(1-(K\lambda)^l)p(x_{m-1},x_m) 
\end{equation}

Now, lettting $l\to \infty $ in \eqref{hence2}, \eqref{estimate2} follows.

The proof is complete.
\end{proof}

It is possible to force the uniqueness of the fixed point of a weak contraction, by imposing an additional contractive condition, quite similar
to \eqref{weakcontraction1}, as shown by the next theorem.

\begin{theorem}\label{main2}
	Let $(X, p,K)$ be a complete partial $K$-metric space and $T : X \to X$
	a weak contraction, i.e. a mapping satisfying
	
	\begin{equation}\label{weakcontraction1unique}
	p(T x, T y) \leq \delta · p(x, y) + L_1 p(x,Tx) , \quad 
	\text{for all} x, y \in X .
	\end{equation}

	 \eqref{weakcontraction1} with $\delta  \in (0, 1)$ and
	some $L_1 \geq 0$ such that $\delta+2L_1<1/K$. Then

	\begin{enumerate}
		\item[(a)] $F(T)=\{x^*\}$, i.e. $T$ has a unique fixed point and $p(x^*,x^*)=0$;
		\item[(b)] For any $x_0 \in X$, the Picard iteration given by $(T^nx_0)_{n\geq 0}$ converges to some $x^* \in F (T )$;
		\item[(c)] The following estimates
		
		\begin{equation}\label{estimate1}
		p(x_m,x^*) \leq \frac{K\lambda^m}{1-K\lambda}p(x_0,x_1), \qquad m=0,1,2,\cdots
		\end{equation}	
		
		\begin{equation}\label{estimate2}
		p(x_m,x^*) \leq \frac{K\lambda}{1-K\lambda}p(x_{m-1},x_m), \qquad m=1,2,\cdots
		\end{equation}
		hold;	
		
		\item[(d)] The rate of convergence of the Picard iteration is given by
		
		\begin{equation}\label{estimation}
			 p(x_n,x^*)\leq \delta p(x_{n-1},x^*)
		\end{equation}

	\end{enumerate}
\end{theorem}

\begin{proof}
	Assume $T$ has two distinct fixed points $x^* , y^* \in X$. We know that $p(x^*,x^*)=0=p(y^*,y^*)$. Then by \eqref{weakcontraction1unique}, we get 
	
	\[  p(Tx^* , Ty^*)= p(x^* , y^*) \leq  \delta p(x^*,y^*)   + L_1p(x^*,x^*) =  \delta p(x^*,y^*) \Longleftrightarrow (1-\delta) p(x^* , y^*) \leq 0,    \]
	
	so contradicting $p(x^*,y^*)  > 0.$
	
	Moreover, putting $y:=x_n,x:=x^*$ in \eqref{weakcontraction1unique}, we obtain we obtain the estimate \eqref{estimation}.
	
	The rest of proof follows by Theorem \ref{main1}.
\end{proof}

We conclude this section by this immediate implication of Theorem \ref{main1}, which is its reformulation in the asymmetric setting.

\begin{theorem}
Let $(X, p,K)$ be a $\tau(p)$-bicomplete $K$-partial quasi-metric space and $T : X \to X$
a $p$-weak contraction, i.e. a mapping satisfying 

\begin{equation}\label{weakcontractionasym}
p(T x, T y) \leq \delta · p(x, y) + Lp(Tx, y) , \quad 
\text{for all } x, y \in X .
\end{equation}
with $\delta \in (0, 1)$ and some $L \geq 0$ with $\delta +2L<1/K$. Then

\begin{enumerate}
	\item[(a)] $F (T ) = \{x \in X : T x = x\} \neq \emptyset$;
	\item[(b)] For any $x_0 \in X$, the Picard iteration given by $(T^nx_0)_{n\geq 0}$ converges to some $x^* \in F (T )$.
\end{enumerate}
\end{theorem}

The proof is left to the reader, as it can easily be retried, following the steps of the proof in the cases of the $p$-Banach, $p$-Kannan and $p$-Chatterjea contractions.

\section{From $K$-partial quasi-metrics to $K$-quasi-metrics}
In this section, we seek a way to formulate fixed point theorems from $K$-partial quasi-metric to $K$-quasi-metrics. 
First, we introduce the following additional notions on a $K$-partial quasi-metric spaces.

\begin{definition}\label{deff2}
	
	Let $(X, p,K)$ be a $K$-partial quasi-metric space with coefficient $K$. Let $(x_n)_{n\geq 1}$ be any sequence in $X$ and $x \in X$. Then:
	\begin{enumerate}
		\item The sequence $(x_n)_{n\geq 1}$ is said to be a $\tau(p)$-0-Cauchy sequence if
		$$\lim\limits_{n<m,n,m\to \infty} p(x_n , x_m)=0.$$ 
		
		$(X, p,K)$ is said to be $\tau(p)$-0-complete if for every $\tau(p)$-0-Cauchy
		sequence $(x_n)_{n\geq 1} \subseteq X$, there exists $x \in X$ such that:
		
		$$ \lim\limits_{n<m,n,m\to \infty} p(x_n , x_m)= \lim\limits_{n\to \infty} p(x_n , x)=p(x,x)=0.$$
		
	\end{enumerate}
\end{definition}

The relation between $\tau(p)$-completeness and $\tau(p)$-0-completeness of a $K$-partial quasi-metric space
is as follows.

\begin{lemma}\label{lemma2.2}
	Let $(X, p,K)$ be a $K$-partial quasi-metric space. If $(X, p,K)$ is $\tau(p)$-complete, then it is $\tau(p)$-0-complete.
\end{lemma}

\begin{proof}
	Let $(x_n)_{n\geq 1}$ be a $\tau(p)$-0-Cauchy
	sequence. Then $\lim\limits_{n<m,n,m\to \infty} p(x_n , x_m)=0.$ This proves 
	proves that $(x_n)_{n\geq 1}$ is a $\tau(p)$Cauchy sequence in $(X, p,K)$. Since $(X, p,K)$ is $\tau(p)$-complete, there exists $x^* \in X$ such that
	
	\[ \lim\limits_{n<m,n,m\to \infty} p(x_n , x_m)= \lim\limits_{n\to \infty} p(x_n , x^*)=p(x^*,x^*). \]
	
	Since $\lim\limits_{n<m,n,m\to \infty} p(x_n , x_m)=0,$ then
	
	\[ \lim\limits_{n<m,n,m\to \infty} p(x_n , x_m)= \lim\limits_{n\to \infty} p(x_n , x^*)=p(x^*,x^*)=0. \]
This proves that $(X, p,K)$ is $\tau(p)$-0-complete.	
	
\end{proof}

The converse of Lemma \ref{lemma2.2} does not hold as shown in the following example.

\begin{example}
	Let $X = (0, 1)$ and $p(x, y) = \max\{y-x,0\} + 1$ for all $x, y \in X$. Then $(X, p,K)$
	is a $\tau(p)$-0-complete, $K$-partial quasi-metric space with coefficient $K = 1$. Since 
	
	\[ \lim\limits_{n<m,n,m\to \infty} p(x_n , x_m) = \lim\limits_{n<m,n,m\to \infty} \left(\max \left\lbrace \frac{1}{m}-\frac{1}{n},0   \right\rbrace +1\right) =1, \]
	
	hence $(x_n)_{n\geq 1}$ is a $\tau(p)$-Cauchy sequence in $(X, p,K)$. By the way of contradiction, assume there exists $x^*\in X$ such that $\lim\limits_{n\to \infty} p(x_n , x^*)=p(x^*,x^*).$ Therefore, 
	
	\[ \lim\limits_{n\to \infty} p(x_n , x^*) =  \lim\limits_{n\to \infty}\left(\max \left\lbrace x^*-\frac{1}{n},0   \right\rbrace +1\right)=p(x^*,x^*)=1, \]
	
	which implies that $x^* \leq 0$. It is a contradiction since $(-\infty,0]\cap X=\emptyset$.
\end{example}

Now we state the relation between a $K$-partial quasi-metric  and certain $K$-quasi-metric as follows:

\begin{theorem}\label{theorem2.4}
	Let $(X, p,K)$ be a $K$-partial quasi-metric space with coefficient $K\geq 1$. For all $x, y \in X$, put

	$$
	d_p(x,y)=
	\begin{cases}
	0 \qquad \text{ if } x = y,\\
	p(x, y) \text{ if } x \neq y.
	\end{cases}
	$$
	
Then we have

	\begin{enumerate}
	
	\item $d_p$ is a $K$-partial quasi-metric with coefficient $K$ on $X$.
	
	\item If $\lim\limits_{n\to \infty} d_p(x_n,x)=0$, so $\lim\limits_{n\to \infty} p(x_n,x)=p(x,x).$
	
	\item $(X, p,K)$ is $\tau(p)$-0-complete if and only if $(X, d_p )$ is left $K$-complete.
	\end{enumerate}	
		
\end{theorem}

\begin{proof}

\vspace*{0.2cm}

\begin{enumerate}
	
	\item We have $d_p$ is a function from $X \times X \to [0,\infty)$. Moreover, $d_p (x, y) = 0$ if and
	only if $x = y$.
	
	For all $x, y, z \in X$, if $x = y$ or $y = z$ or $z = x,$ then $d_p (x, y) \leq d_p (x, z) + d_b (z, y).$
	
If $x \neq y \neq z,$ then	
	
	\begin{align*}
		d_p (x, y) = p(x, y) &\leq K [p(x, z) + p(z, y)] -p(z, z)\\
		&\leq K [p(x, z) + p(z, y)] = K [d_p (x, z) + d_p (z, y)] .
	\end{align*}
	Hence $d_p$ is a $K$-quasi-metric with coefficient $K$ on $X$.
	
	\item If there exists $n_0$ such that $x_n = x$ for all $n \geq n_0$ , then $\lim\lim\limits_{n\to \infty} p(x_n , x) = p(x, x).$ So we may assume that $x_n\neq x$ for all $n\in \mathbb{N}$. Then $d_p(x_n , x) = p(x_n , x)$ for all $n\in \mathbb{N}$. Since $\lim\limits_{n\to \infty} d_p(x_n,x)=0,$ we have $\lim\limits_{n\to \infty} d_p(x_n,x)=\lim\limits_{n\to \infty} p(x_n,x)=0.$ 
	Moreover, by $0 \leq p(x, x) \leq p(x_n , x)$ for all $n\in \mathbb{N}$, we have $0 \leq p(x, x) \leq \lim\limits_{n\to \infty} p(x_n , x) = 0$. This proves that $\lim\limits_{n\to \infty} p(x_n,x)=0=p(x,x).$
	In conclusion, $\lim\limits_{n\to \infty} p(x_n,x)=p(x,x).$

	\item ($\Longrightarrow$). Let $(x_n)_{n\geq 1}$ be a left $K$-Cauchy sequence in $(X, d_p )$. Then $\lim\limits_{n<m,n,m\to \infty} p(x_n,x_m)=0.$
 If there exists $n_0$ such that $x_n = x$ for all $n \geq n_0$, then $\lim\limits_{n\to \infty} d_p(x n,x) = 0.$ 	
	So we may assume that $x_n\neq x_m$ for all $n\neq m$. It implies that
	
	\[ \lim\limits_{n<m,n,m\to \infty} p(x_n,x_m)= \lim\limits_{n<m,n,m\to \infty} d_p(x_n,x_m)=0.   \]

Then $(x_n)_{n\geq 1}$ is a $\tau(p)$-0-Cauchy sequence in $(X, p,K)$. Since $(X, p,K)$ is $\tau(p)$-0-complete, there exists
$x \in X$ such that	

\[  \lim\limits_{n<m,n,m\to \infty} p(x_n , x_m)= \lim\limits_{n\to \infty} p(x_n , x)=p(x,x)=0.\]
	
Note that $0 \leq d_p (x_n , x) \leq p(x_n , x)$ for all $n\in \mathbb{N}$, then

	\[ 0 \leq \lim\limits_{n\to \infty} d_p(x_n , x)\leq \lim\limits_{n\to \infty} p(x_n , x)=0.  \]
	
	Therefore $\lim\limits_{n\to \infty} d_p(x_n , x) =0$ and $(X, d_p )$ is left $K$-complete.

	($\Longleftarrow$). Let $(x_n)_{n\geq 1}$ be a $\tau(p)$-0-Cauchy sequence in $(X, p,K)$, this means that $\lim\limits_{n<m,n,m\to \infty} p(x_n , x_m)=0.$ Since $0 \leq d_p(x_n , x_m ) \leq p(x_n , x_m )$ for all $n,m\in \mathbb{N}$, we have $\lim\limits_{n<m,n,m\to \infty}d_p(x_n , x_m ) =0.$ This proves that $(x_n)_{n\geq 1}$ is a left $K$-Cauchy sequence in $(X, d_p )$. Since $(X, d_p )$ is left $K$-complete, there exists $x \in X$ such that$ \lim\lim\limits_{n\to \infty} d_p(x_n , x) = 0.$ If there exists $n_0$ such that $x_n = x$ for all $n \geq n_0$, then

	\[ \lim\limits_{n<m,n,m\to \infty}p(x_n , x_m ) =\lim\limits_{n\to \infty}p(x_n , x )=p(x,x).          \]
	
	Since $\lim\limits_{n<m,n,m\to \infty}p(x_n , x_m )=0,$ we get $$\lim\limits_{n<m,n,m\to \infty}p(x_n , x_m ) =\lim\limits_{n\to \infty}p(x_n , x )=p(x,x)=0.$$
	
	So we may assume that $x_n \neq x_m$ whenever $n\neq m.$. Then $\lim\limits_{n\to \infty}p(x_n,x)= \lim\limits_{n\to\infty}d_p(x_n,x)=0$. Moreover, in view of $0\leq p(x,x)\leq p(x_n,x)$, $0\leq p(x,x)\leq \lim\limits_{n\to \infty}p(x_n,x) =0,$ that is $p(x,x)=0.$ Therefore, we also have
	
	\[ \lim\limits_{n<m,n,m\to \infty}p(x_n , x_m ) =\lim\limits_{n\to \infty}p(x_n , x )=p(x,x)=0. \]
	
We conclude that $(X,p,K)$ is $\tau(p)$-0-complete.
	
\end{enumerate}

\end{proof}

The following example shows that the converse of statement 2 from Theorem \ref{theorem2.4}
does not hold.

\begin{example}
	Consider the $1$-partial quasi-metric $p(x,y) = \max\{ y-x,0\} + 1$ for all $x,y \in X = [0,1]$.
	From 
	
	\[ \lim\limits_{n\to \infty}p\left( \frac{1}{n} ,0\right)=1 =p(0,0) .\]
	
	This entails that $\left(\frac{1}{n}\right)_{n\geq 1}$ is $\tau(p)$-convergent to $0$ in $(X,p,1)$.
	
	On the other hand, we have 
	\[ \lim\limits_{n\to \infty}d_p\left( \frac{1}{n} ,0\right)= \lim\limits_{n\to \infty} \left( -\frac{1}{n} +1\right) = 1\neq 0. \]
	This proves that $\left(\frac{1}{n}\right)_{n\geq 1}$ is not left $K$-convergent to $0$ in $(X,d_p,1)$.

\end{example}

We conclude this section by giving relation between contraction conditions on $K$-partial quasi-metric spaces in  and
certain contraction conditions on $K$-quasi-metric spaces is as follows. More precisely, we reformulate the $p$-Banach and the $p$-Kannan contractions in terms of $K$-quasi-metric spaces.

\begin{theorem}
	Let $(X,p,K)$ be a $K$-partial quasi-metric space with coefficient $K$ and $d_p$ be as defined in Theorem \ref{theorem2.4}, and $T : X \to X$ be a map. Then we have:
	
	\begin{enumerate}
		\item If there exists a constant $0\leq \lambda<1$ satisfying
		$$p(T x, T y) \leq \lambda p(x, y)
		\quad \text{ for all } x, y \in X,$$
		
		then

		$$d_p(T x, T y) \leq \lambda d_p(x, y)
		\quad \text{ for all } x, y \in X.$$

		\item If there exists a constant $,0\leq \lambda < 1/2$ satisfying
		$$p(T x, T y) \leq \lambda [p(Tx, x) + p(y,Ty)]
		\quad \text{ for all } x, y \in X,$$
		
		then 
		$$d_p(T x, T y) \leq \lambda [d_p(Tx, x) + d_p(y,Ty)]
		\quad \text{ for all } x, y \in X.$$
		
		\item If there exists $\lambda$ such that 
		$$p(T x, T y) \leq \lambda \max \{p(x, y), p( T x,x), p(y, T y)\} \quad \text{ for all } x\neq y \in X$$
		 then $$d_p(T x, T y) \leq \lambda \max \{d_p (x, y), d_p (T x,x), d_p (y, T y)\}$$ for all $x, y \in X.$
		
	\end{enumerate}
\end{theorem}

\begin{proof}
\begin{enumerate}
	
	\item If $x=y$, then $d_p(x,y) = 0 \leq \lambda d_p(x,y)$. If $x\neq y$, then $d_p(x,y) = p(x,y)$ and we have 
	\[d_p(Tx,Ty) \leq p(Tx,Ty) \leq \lambda p(x,y)=\lambda d_p(x,y) \]. Therefore $$d_p(T x, T y) \leq \lambda [d_p(Tx, x) + d_p(y,Ty)]
	\quad \text{ for all } x, y \in X.$$

	\item If $x=Tx,$ then 
	
	$$p(x,Tx)=p(Tx,Tx) \leq \lambda [p(x,Tx)+p(Tx,x)] = 2 \lambda p(x,Tx).$$
	
	Since $2\lambda \in [0, 1)$ , we have $p(x, T x)=p(Tx,x) = 0 = d_p (x, T x)=d_p(Tx,x).$ It implies that $p(x,Tx) = d_p(x,Tx)$ and $p(Tx,x) = d_p(Tx,x)$ for all $x\in X.$ Therefore, for all $x,y\in X$

	\[ d_p(Tx,Ty) \leq p(Tx,Ty) \leq \lambda [p(Tx, x) + p(y,Ty)]= \lambda [d_p(Tx, x) + d_p(y,Ty)].          \]
	
	\item For all $x, y \in X$, we have
	
	\begin{equation}\label{eq3}
	 \max \{d_p (x, y), d_p (T x,x), d_p (y, T y)\} \leq \max \{p(x, y), p(T x,x), p(y, T y)\} .
	 \end{equation}

	In order to prove that
	
	\begin{equation}\label{eq4}
	\max \{p(x, y), p(T x,x), p(y, T y)\} \leq \max \{d_p (x, y), d_p (T x,x), d_p (y, T y)\}
	\end{equation}
	
for all $x \neq y \in X$, we distinguish between two cases.	

\vspace*{0.3cm}

\textbf{Case 1}. There exist $x, y \in X$ such that $\max \{p(x, y), p(T x,x), p(y, T y)\} = p(x, y).$
Since $p(x, y) = d_p (x, y)$, we see that \eqref{eq4} holds.

\vspace*{0.3cm}

\textbf{Case 2}. There exist $x, y \in X$ such that $\max \{p(x, y), p(x, T x), p(y, T y)\} = p(T x,x).$ If $x = T x$, then $p(x, T x) = p(x, x) \leq p(x, y) = d_p (x, y).$ Therefore, \eqref{eq4} holds. If $x \neq T x$, then $p(x, T x) = d_p (x, T x)$. It also implies that \eqref{eq4} holds.
By the above two cases, we see that \eqref{eq4} holds for all $x \neq y$. It follows from \eqref{eq3} and \eqref{eq4} that, for all $x \neq y$,

$$\max \{d_p (x, y), d_p (T x,x), d_p (y, T y)\} = \max \{p(x, y), p(T x,x), p(y, T y)\} .$$

Therefore,

\begin{align*}
	d_p (T x, T y) &\leq p(T x, T y) \leq \lambda \max \{p(x, y), p(T x,x), p(y, T y)\} \\
	 &\leq \lambda \max \{d_p (x, y), d_p (T x,x), d_p (y, T y)\}
\end{align*}

for all $x \neq y$. If $x = y$, we have $d_p (T x, T y) = 0.$ Then
	
	\[  d_p (T x, T y) \leq \lambda \max \{d_p (x, y), d_p( T x,x), d_p (y, T y)\}  \]
	
	for all $x, y \in X.$
\end{enumerate}
\end{proof}

\begin{remark}
	From Definitions \ref{deff1} and \ref{deff2}, it is easy to define an appropriate notion of $\tau(p)$-0-bicompleteness and to compare $\tau(p)$-0-bicomplete $K$-partial quasi-metric spaces and bicomplete $K$-quasi-metric spaces.
\end{remark}
	
	 Moreover, the author intends, in \cite{gabx}, to give alternative proofs of Theorem \ref{p-banach} and Theorem \ref{p-kannan}. We also do believe that Inequality \ref{maxcond} can reformulated in terms of $K$-quasi-metric spaces and we intend to take up this investigation as well.

In concluding this section, we would like to ask the following questions:

\begin{itemize}
	\item If $T : X \to X$ is a contraction with respect to a $K$-partial quasi-metric $p$, which
	conditions does $T$ satisfy with respect to $d_p$ (as defined in Theorem \ref{theorem2.4} ) ?
	\item How to use fixed point theorems in a quasi-metric type space $(X, d_p )$ to give analogous fixed point
	results in a $K$-partial quasi-metric space?
\end{itemize}

\section{Weighted $K$-quasi-metrics}

The idea of weight function, introduced by K\"unzi et al. \cite{kun} for quasi-metrics, can easily be extended to $K$-quasi-metrics and we can derive new results for the theory of $K$-partial quasi-metrics. The corresponding theory for quasi-metrics with weight can be read in \cite{kun} and the results we present are merely copies of the ones already obtained by K\"unzi et al.\cite{kun}.
More precisely, we have:

\begin{definition}
	An arbitrary $K$-quasi-metric space $(X, D,K)$ equipped with an arbitrary (so-called weight) function $w : X \to
	[0, \infty)$ will be called a $K$-quasi-metric space with weight $w$ or a weighted  $K$-quasi-metric space with weight $w$.
\end{definition}

We describe a bijection between $K$-quasi-metrics with weight and lopsided $K$-partial quasi-metrics on $X$ that will be used throughout this section. 

In the following we shall refer to this correspondence often by the (lopsided) $K$-partial
quasi-metric associated with a given $K$-quasi-metric with weight and similar self-explanatory expressions.

\begin{proposition}Compare \cite[Remark 4.]{kun}\label{correspondence}
	
	If $p$ is a lopsided $K$-partial quasi-metric on $X$, then $q(x, y) = p(x, y)-p(x, x)$ whenever $x, y \in X$ and $w(x) = p(x, x)$
	whenever $x \in X$ yield a $K$-quasi-metric space $(X, q)$ with weight $w$, which we denote by $(X, q, w)$.

	If $(X, q, w)$ is a $K$-quasi-metric space with weight, then $p(x, y) = q(x, y) + w(x)$ whenever $x, y \in X$ is a lopsided
	$K$-partial quasi-metric on $X$.
\end{proposition}

Next we define a compatibility condition between $K$-quasi-metric and weight.

\begin{definition}
	A $K$-quasi-(pseudo)metric space type with compatible weight on a set $X$ is a 4-tuple $(X, q ,K,w)$ where $q: X \times X \to
	[0, \infty)$ is a $K$-quasi-(pseudo)metric on $X$ and $w: X \to [0, \infty)$ is a function satisfying $w(y) \leq q(x, y) + w(x)$ whenever
	$x, y \in X$.
\end{definition}

As already observed in Proposition \ref{correspondence}, for each $K$-partial quasi-metric $p$ on a set $X$, we have its associated $K$-quasi-metric $p(x, y) - p(x, x)$ where $x, y \in X$. However, it is good to point out that the  $K$-quasi-metric $p(x, y) - p(y, y)$ where $x, y \in X$ is also deducted from the above correspondence. Hence a quasi-metric space $(X, q,K, w)$ with weight $w$ has a compatible weight on $X$ if and only
if $\hat{q}$ defined by $\hat{q}(x, y) = q(x, y) + w(x) - w(y)$ whenever $x, y \in X$ is a $K$-quasi-metric on $X$.


\begin{proposition}(Compare \cite[Example 2 ]{kun})
Let $(X, q, w)$ a $K$-quasi-metric space with weight, then $(X, q' ,K,w)$ where $q'(x,y) = q(x, y) + \max\{w(y)-w(x),0 \}$, whenever $x,y\in X$, is a $K$-quasi-(pseudo)metric space type with compatible weight.
\end{proposition}	
	
	\begin{proof}
		Indeed, since $w(y)-w(x) \leq \max\{w(y)-w(x),0 \} $, we have
		
		\[ w(y)-w(x) \leq q(x, y) + \max\{w(y)-w(x),0 \} \iff  w(y) \leq q(x, y) + \max\{w(y)-w(x),0 \} +w(x).     \]
	\end{proof}
	
The following is immediate, form the definitions of a weight.

\begin{proposition}(Compare \cite[Remark 7, Proposition 1 ]{kun})
If $w$ is a weight compatible with a given $K$-quasi-metric on $X$, then for any non-negative constant $k$, $w(x) + k$ whenever $x \in X$, is also a compatible weight.
Moreover, if $(X, q,K, w)$ is a $K$-quasi-metric space with compatible weight $w$, then $(X, q^{-1},K, w')$, where $w' = \frac{1}{1+w}$, is a $K$-quasi-metric space with compatible weight $w'$.
\end{proposition}

More generally, the following provides us with a simple transformation to obtain new $K$-quasi-metric space with compatible weight $w$ from old ones.

\begin{proposition}(Compare \cite[Proposition 2 ]{kun})
	If $(X, q, K,w)$ is a $K$-quasi-metric space with compatible weight, then $(X, \hat{q}, K,\hat{w})$ where 
	\begin{enumerate}
		\item 
		$\hat{q}(x, y) = \min\{q(x, y), 1\}$ and $\hat{w}(x) = \min\{w(x), 1\}$ whenever $x, y \in X$ and
		
		\item $$\hat{q}(x, y) = \frac{q(x,y)}{1+q(x, y)} \quad \text{ and } \hat{w}(x) = \frac{w(x)}{1+w(x)}$$ whenever $x, y \in X$ 
	\end{enumerate}
is also a $K$-quasi-metric space with compatible weight.
\end{proposition}

\section{Conclusion and future work}
 
There are numerous generalizations of quasi-metric spaces and a lot of fixed point results have
been obtained, generally by symmetrization, using the results from the metric case. In most cases, 
fixed point results on these new spaces appear to be redundant, although it is not so easy to transfer
the given contractive condition to the new setting. Many authors still argue today, regarding the relevance of "partial metric spaces" introduced by Matthews. The present article builds on the properties of $K$-partial metric spaces and their relations with metric type spaces. 

Hence, two important questions arise naturally:

\begin{itemize}
	\item Is the new setting of $K$-partial quasi-metric spaces topologically equivalent to that of quasi-metric type spaces or to a previously known asymmetric distance function?
	
	\item Can fixed point theorems on $K$-partial quasi-metric space be directly (or easily) obtained from
	fixed point theorems on a quasi-metric type space?
\end{itemize}

It is our belief that,  these new spaces bring a variety of tools and environments where concrete applications could be obtained.

We finish by saying a few words on weak $K$-partial quasi-metric. In 1999, by omitting the small self-distance axiom of partial metric, Heckmann \cite{heck} defined a weak partial metric as a generalization of partial metric. By omitting the small self-distance axioms in Definition \ref{maindef}, we introduce the so-called ``weak $K$-partial metric". In a weak
$K$-partial quasi-metric space, the convergence of a sequence, Cauchy sequence and completeness are defined as in a $K$-partial quasi-metric space. In particular, may results can be recovered from the theory of $K$-partial quasi-metrics. It is easy to establish that:

\begin{proposition}(Compare \cite[Proposition 2.3]{isak})
	Let $(X, p)$ be a weak $1$-partial quasi-metric space. Then $d_w$ defined by
	
	\[ d_w (x, y) = p(x, y) - \min \{p(x, x), p(y, y)\} \]
	
	for all $x, y \in X$, is a quasi-metric on $X$ .
\end{proposition}

\begin{lemma}(Compare \cite[Lemma 2.4]{isak} Let $(X, p)$ be a weak partial quasi-metric space. Then
	$(X, p)$ is complete if and only if $(X, d_w )$ is complete.
\end{lemma}

We intend to take up this investigation in more details in \cite{gabx2}. We shall see that fixed point theorems on weak $K$-partial quasi-metric spaces
may be obtained from fixed point theorems on $K$-partial quasi-metric spaces, and then fixed point theorems on weak $K$-partial metric spaces may be obtained from fixed point
theorems on quasi-metric type spaces.
There, we also introduce the so called ``weak weight functions" and the idea of ``weak partial quasi-metric space with compatible weak weight" to study the completeness for weak partial quasi-metric spaces via Caristi's type mappings.

	\bibliographystyle{amsplain}

\end{document}